\documentclass{actams-author}
\usepackage{color}
\numberwithin{equation}{section}

\begin{document}
	
	\ID{E13-xxx}
	
	\DATE{Final, 201x-xx-xx}
	
	\PageNum{1}
	
	\Volume{201x}{}{3x}{x}
	
	\EditorNote{$^*$ Received x x, 201x; revised x x, 201x. This research is supported by the National Key R$\&$D Program of China (No. 2020YFA0712901).}
	
	\abovedisplayskip 6pt plus 2pt minus 2pt
	\belowdisplayskip 6pt plus 2pt minus 2pt

	\def\vsp{\vspace{1mm}}
	\def\th#1{\vspace{1mm}\noindent{\bf #1}\quad}
	\def\proof{\vspace{1mm}\indent{\bf Proof}\quad}
	\def\no{\nonumber}
	
	\def\q{\quad}\def\qq{\qquad}
	
	\newenvironment{prof}[1][Proof]{\indent\textbf{#1}\quad}{\hfill $\Box$\vspace{0.7mm}}
	
	\allowdisplaybreaks[4]
	
	\theoremstyle{plain}
	\newtheorem{condition}[theorem]{Condition}
	
	\def\blue{\color{blue}}\def\red{\color{red}}

	\AuthorMark{Z. Cheng\,\& Z. Li}
	
	\TitleMark{\uppercase{Remaining-lifetime age-structured branching processes}}
	
	\title{\uppercase{Remaining-lifetime age-structured branching processes}
		\footnote{ }}
	
		
		\author{\sl{Ziling \uppercase{Cheng}}} 
		{School of Mathematical Sciences, Beijing Normal University, Beijing {\rm100875}, People's Republic of China\\
			E-mail\,$:$ zlcheng@mail.bnu.edu.cn}
		
		\author{\sl{Zenghu \uppercase{Li}}}
		{School of Mathematical Sciences, Beijing Normal University, Beijing {\rm100875}, People's Republic of China\\
			E-mail\,$:$ lizh@bnu.edu.cn}

		\maketitle
		
	
	\Abstract{We study age-structured branching models with reproduction law depending on the remaining lifetime of the parent. The lifespan of an individual is decided at its birth and its remaining lifetime decreases at the unit speed. The models without or with immigration are constructed as measure-valued processes by pathwise unique solutions of stochastic equations driven by time-space Poisson random measures. In the subcritical branching case, we give a sufficient condition for the ergodicity of the process with immigration. Two large number laws and a central limit theorem of the occupation times are proved.}
	
	
	\MRSubClass{60J80, 60J85, 60H15, 60F15, 60F05}
	
	\section{Introduction}\label{section1}
	
	The oldest model for the stochastic evolution of a population is the Bienaym\'{e}-Galton-Watson process, which assumes an individual has unit lifespan and gives birth to offspring at the end of its life. It is a discrete-time Markov chain introduced by Bienaym\'{e} \cite{Bie45} and Galton and Watson \cite{Watson74}. For a classical continuous-time branching process, an individual is assumed to have exponentially distributed  lifespan and split independently of the others on its death. These assumptions were decided by mathematical convenience. From the viewpoint of applications, it is reasonable to consider more general assumptions on the lifespan and branching mechanism. The age-dependent branching process introduced by Bellman and Harris \cite{Bellman52} assumes that the individual may have a general lifespan distribution, but children are only born when a parent dies. General branching processes allowing births throughout the whole life of the parent were introduced by Crump and Mode \cite{Crump68} and Jagers \cite{Jagers69}; see also \cite{Crump69, Doney72a, Doney72b, Jagers89, Kendall49}. Those models can be formulated in terms of measure-valued Markov processes, although the evolutions of the population sizes are usually not Markovian. In fact, the Markov property remains even when the lifespan and the reproduction are allowed to depend on the whole population; see, e.g., \cite{Hamza13, Jagers00, Jagers11, Metz13, Oelschlager90, Tran08}. Several authors have also studied population models where the reproduction of an individual may depend on its age; see, e.g., \cite{Bose95, Bose86, Bose00, Dawson02, Kaj98}.
	
	Stochastic equations have been playing increasingly important roles in recent developments of the theory of branching processes. For their applications to discrete-state branching models, the reader may refer to \cite{CFM07, FM04, Fan20, Ji21, Tran08}, where the individuals were allowed to give birth to offspring at death. The approach has also played an important role in the study of continuous-state branching processes; see, e.g., \cite{DaL06, DaL12, HLY14, Li20, Li22, Par16, Xiong13}. In particular, several types of measure-valued branching processes have been constructed in \cite{DaL12, HLY14, Li22, Par16, Xiong13} by the pathwise unique solutions to stochastic equations driven by time-space noises.

	The occupation times give useful perspectives to the study of population models. The longtime behavior of such times is an interesting topic of research and has been investigated by several authors. In particular, Iscoe \cite{Iscoe86a} defined the occupation time process for a measure-valued branching process and proved several associated central limit theorems; see also \cite{Dawson01, Hong02, Iscoe86b, Li99, Tang06}.

	In this paper, we study a special class of age-structured branching processes, where the lifespan of an individual is distributed arbitrarily and determined at its birth. The remaining lifetime of an individual decreases at the unit speed. Furthermore, we allow the birth rate and the reproduction law to depend on the remaining lifetime of the parent. The model is a special case of the age-structured branching process introduced by Crump and Mode \cite{Crump68} and Jagers \cite{Jagers69}. It can also be thought of as a typical special non-local branching particle systems; see, e.g., \cite{Bose00, Bose95, Dawson02, Kaj98, Li22}. We give constructions of the models without or with immigration as measure-valued processes by pathwise unique solutions to stochastic integral equations driven by time-space Poisson random measures. In particular, we obtain the martingale problems associated with the processes as an application of the stochastic equations. In the subcritical branching case, we give a sufficient condition for the ergodicity of the process with immigration. Moreover, we prove two large number laws and a central limit theorem of the occupation times of the process.

	Let $\mathbb{N}=\{0,1,2,\ldots\}$. Let $\mathfrak{M}(0,\infty)$ be the set of finite Borel measures on $(0,\infty)$ with the weak convergence topology. Let $\mathfrak{D}(0,\infty)$ be the set of bounded positive right-continuous increasing functions $f$ on $\mathbb{R}$ satisfying $f(x)=0$ for $x\le 0$. We identify $\nu\in\mathfrak{M}(0,\infty)$ with its distribution function $\nu\in\mathfrak{D}(0,\infty)$ defined by $\nu(x)=\nu(0, x]$ for $x>0$. Let $\mathfrak{N}(0,\infty)$ be the subset of $\mathfrak{M}(0,\infty)$ consisting of integer-valued measures. Let $\mathscr{B}(0,\infty)$ denote the Borel $\sigma$-algebra on $(0,\infty)$. Let $B(0,\infty)$ be the Banach space of bounded Borel functions on $(0,\infty)$ furnished with the supremum norm $\|\cdot\|$. Let $C(0,\infty)$ be the set of continuous functions in $B(0,\infty)$. Let $C^{1}(0,\infty)$ be the set of functions in $C(0,\infty)$ with bounded continuous derivatives of the first order. We use the superscript ``+'' to denote the subsets of positive elements and use the subscript ``0'' to denote the subsets of functions vanishing on $\{0\}\cup\{+\infty\}$, e.g., $B(0,\infty)^{+}$, $C_{0}(0,\infty)^{+}$, etc. For any function $f$ on $A\subset\mathbb{R}$, we understand that $f(x)=0$ for $x\in\mathbb{R}\backslash A$ by convention. For any $f\in B(0,\infty)$ and $\nu\in\mathfrak{M}(0,\infty)$ write $\langle\nu,f\rangle= \int_{(0,\infty)}f(x)\nu(d x)$. In the integrals, we make the convention that, for $a\le b\in\mathbb{R}$,
	\begin{align*}
		\int_{a}^{b}=\int_{(a, b]}\quad \text { and } \quad \int_{a}^{\infty}=\int_{(a, \infty)}.
	\end{align*}

	\section{An age-structured branching process}\label{section2}
	
	Let $\alpha\in C^{1}(0,\infty)^{+}$. For each $x\in (0,\infty)$, let $\{p(x, i):i\in\mathbb{N}\}$ be a discrete probability distribution with generating function
	\begin{align*}
		g(x,z)=\sum\limits_{i=0}^{\infty}p(x,i) z^{i}, \quad z \in[0,1].
	\end{align*}
	Then for any $x\in(0,\infty)$ let $g'(x,z)=\frac{\partial}{\partial z} g(x,z)$ and $g''(x,z)=\frac{\partial^2}{\partial z^2} g(x,z)$ denote respectively the first and second derivative with respect to $z$ of $g(x,z)$. Throughout the paper, we assume that $p(\cdot,i) \in C^{1}(0,\infty)^{+}$ for every $i \in \mathbb{N}$ and
	\begin{align}\label{2.1}
		\|g'(\cdot, 1-)\|=\sup _{x>0} \sum_{i=1}^{\infty} p(x, i) i<\infty,
	\end{align}
	Then we have $g(\cdot, z)\in C^{1}(0,\infty)^{+}$ for each $z\in [0,1]$. A branching particle system is characterized by the following properties:
	
	\begin{description}
		
		\item[\textmd{(i)}] The remaining lifetimes of the particles decrease at the unit speed, i.e., they move according to realizations of the deterministic process $\xi=(\xi_{t}\vee 0)_{t \ge 0}$ in $(0,\infty)$ defined by $\xi_{t}=\xi_{0}-t$.
		
		\item[\textmd{(ii)}] A particle gives birth to offspring during its life. For a particle which is alive at time $r \ge 0$ with remaining lifetime $x>t-r>0$, the conditional probability of having not given birth at time $t$ is $\exp\{-\int_{0}^{t-r} \alpha(x-s) d s\}$.
		
		\item[\textmd{(iii)}] When a particle gives birth at remaining lifetime $x > 0$, it firstly gives birth to a random number of offspring according to the probability law $\{p(x, i): i\in \mathbb{N}\}$ determined by the generating function $g(x, \cdot)$, those offspring then choose their life-lengths in $(0,\infty)$ independently of each other according to the continuous probability distribution $G(d t)$.
		
	\end{description}
	
	We assume that the lifetimes and the offspring reproductions of different particles are independent. Let $X_{t}(B)$ denote the number of particles alive at time $t \ge 0$ with remaining lifetimes belonging to the Borel set $B \subset (0,\infty)$. If we assume $X_{0}((0, \infty))<\infty$, then $(X_{t}: t\ge 0)$ is a Markov process with state space $\mathfrak{N}(0,\infty)$. We refer to Li \cite[Section 4.3]{Li22} for the formulation of general branching particle systems.
	
	Let $\sigma \in \mathfrak{N}(0,\infty)$ and let $(X_{t}^{\sigma}: t \ge 0)$ be the above system with initial value $X_{0}^{\sigma}=\sigma$. Let $\delta_{x}$ denote the unit measure concentrated at $x \in (0,\infty)$. Suppose that the process is defined on a probability space $(\Omega, \mathscr{G}, \mathbf{P})$. Notice that the remaining lifetime of a living particle must be greater than $0$, then the above properties imply that
	\begin{align}\label{2.2}
		\mathbf{E}[\exp \{-\langle X_{t}^{\sigma}, f\rangle\}]=\exp \{-\langle\sigma, u_{t} f\rangle\}, \quad f \in B(0, \infty)^{+},
	\end{align}
	where
	\begin{align*}
		u_{t} f(x)=-\log \mathbf{E}\big[\!\exp \{-\langle X_{t}^{\delta_{x}}, f\rangle\}\big].
	\end{align*}
	From properties (i), (ii) and (iii) we derive as in Kaj and Sagitov \cite[Section 3]{Kaj98} and Li \cite[Section 4.3]{Li22} the following renewal equation:
	\begin{align}\label{2.3}
		e^{-u_{t} f(x)}=\ &e^{-f(x-t)-\int_{0}^{t} \alpha(x-s) d s} \nonumber\\
		&+\int_{0}^{t} \alpha(x-s) e^{-\int_{0}^{s} \alpha(x-r) d r-u_{t-s} f(x-s)} g(x-s,\langle G, e^{-u_{t-s} f}\rangle)d s,
	\end{align}
	where $f \in B(0,\infty)^{+}$. This follows as we think about the Laplace functional of the random measure $X_{t}^{\delta_{x}}$ produced by a single particle having life-length $x \in (0,\infty)$. We assume the first offsprings in the particle system are born at time $\eta$. In view of (\ref{2.2}), the Laplace functional of $X_{t}^{\delta_{x}}$ is given by the left-hand side of (\ref{2.3}). We first consider the case where $\eta > t$, i.e., the initial particle has not given birth at time $t$. By the Markov property of $(X_{t}^{\sigma}: t \ge 0)$ we have
	\begin{align*}
		\mathbf{P}\big[1_{\{\eta>t\}} e^{-\langle X_{t}^{\delta_{x}},f\rangle}\big]
		&=\mathbf{P}\big[1_{\{\eta>t\}} e^{-\langle X_{t}^{\delta_{x}},f\rangle}\big]1_{\{x>t\}}\\
		&=\mathbf{P}\big[1_{\{\eta>t\}} \mathbf{P}\big[ e^{-\langle X_{t}^{\delta_{x}},f\rangle}\big|\mathscr{G}_t \big]\big] 1_{\{x>t\}}\\
		&=\mathbf{P}\big[1_{\{\eta>t\}} \mathbf{P}\big[e^{-\langle \delta_{x-t},f\rangle}\big]\big] 1_{\{x>t\}}\\
		&=\mathbf{P}\big[e^{-\langle \delta_{x-t},f\rangle} 1_{\{x>t\}}\big]\mathbf{P}[\eta>t] \\
		&=e^{-f(x-t)-\int_{0}^{t} \alpha(x-s) d s},
	\end{align*}
	which gives the first term on the right-hand side of (\ref{2.3}). By property (iii), if it happens that $0< \eta \le t$, i.e., the initial particle gives birth at remaining lifetime $x-\eta >0$, it firstly produces a random number of offspring according to the probability law $\{p(x-\eta, i): i\in \mathbb{N}\}$ determined by the generating function $g(x-\eta, \cdot)$, those offspring then choose their life-lengths in $(0,\infty)$ independently of each other according to the probability distribution $G(d t)$. Notice that the lifetimes and the offspring reproductions of different particles are independent. With those considerations we compute
	\begin{align*}
		\mathbf{P}\big(1_{\{0<\eta\le t\}} &e^{-\langle X_{t}^{\delta_{x}},f\rangle}\big)
		=
		\mathbf{P}\big[1_{\{0<\eta\le t\}} \mathbf{P}\big(e^{-\langle X_{t}^{\delta_{x}},f\rangle}\big|\mathscr{G}_{\eta} \big)\big]\\
		&=
		\mathbf{P}\Big\{1_{\{0<\eta\le t\}} \sum\limits_{i=0}^{\infty}p(x-\eta,i)\Big[\int_{0}^{\infty}\mathbf{P}(e^{-\langle X_{t-\eta}^{\delta_{y}},f\rangle})G(dy)\Big]^{i}\mathbf{P}\big(e^{-\langle X_{t-\eta}^{\delta_{x-\eta}},f\rangle}\big)\Big\}\\
		&=
		\mathbf{P}\Big[1_{\{0<\eta\le t\}} \sum\limits_{i=0}^{\infty}p(x-\eta,i)\langle G,e^{-u_{t-\eta}f}\rangle^{i}e^{-u_{t-\eta}f(x-\eta)}\Big]\\
		&=
		\mathbf{P}\big[1_{\{0<\eta\le t\}} g(x-\eta,\langle G,e^{-u_{t-\eta}f}\rangle) e^{-u_{t-\eta}f(x-\eta)}\big]\\
		&=
		\int_{0}^{t} \alpha(x-s) e^{-\int_{0}^{s} \alpha(x-r) d r-u_{t-s} f(x-s)} g(x-s,\langle G, e^{-u_{t-s} f}\rangle) d s,
	\end{align*}
	which leads to the second term on the right-hand side of (\ref{2.3}). Furthermore, by Li \cite[Proposition 2.9]{Li22}, the above equation implies
	\begin{align}\label{2.4}
		e^{-u_{t} f(x)}=e^{-f(x-t)}+\int_{0}^{t}\big[g(x-s,\langle G, e^{-u_{t-s} f}\rangle)-1\big] e^{-u_{t-s} f(x-s)} \alpha(x-s) d s,
	\end{align}
	where $f \in B(0, \infty)^{+}$. The uniqueness of the solution to (\ref{2.3}) and (\ref{2.4}) follows by Gronwall's inequality. Then the two equations are equivalent.
	
	We call any Markov process $(X_{t}: t \ge 0)$ with state space $\mathfrak{N}(0,\infty)$ a \textit{remaining-lifetime age-structured branching process with parameters $(\alpha, g, G)$} or simply an \textit{$(\alpha,g,G)$-process} if it has transition semigroup $(Q_{t})_{t \ge 0}$ defined by
	\begin{align}\label{2.5}
		\int_{\mathfrak{N}(0,\infty)} e^{-\langle\nu, f\rangle} Q_{t}(\sigma, d \nu)=\exp \{-\langle\sigma, u_{t} f\rangle\}, \quad f \in B(0,\infty)^{+},
	\end{align}
	where $u_{t} f(x)$ is the unique solution to (\ref{2.4}). For any $t \ge 0$ let $T_{t}$ be the operator on the Banach space $C_{0}(0,\infty)$ defined by $T_{t} f(x)=f(x-t)$. For any $f \in B(0, \infty)^{+}$, by (\ref{2.4}) it is easy to see that $U_{t} f(x)=1-e^{-u_{t} f(x)}$ is the unique bounded positive solution to the evolution integral equation
	\begin{align}\label{2.6}
		U_{t} f(x)=T_{t} U_{0} f(x)-\int_{0}^{t} T_{t-s} \phi(\cdot, U_{s} f)(x) d s,
	\end{align}
	where
	\begin{align*}
		\phi(x, u)=\alpha(x)\big[g(x,\langle G, 1-u\rangle)-1\big][1-u(x)].
	\end{align*}
	
	\begin{proposition}\label{prop2.1} For any $f\in B(0,\infty)^{+}$, the map $(t,x)\mapsto u_t f(x)$ is the unique solution to
		\begin{align}\label{2.7'}
			u_{t}f(x)=f(x-t)+\int_{0}^{t} \alpha(x-s)\big[1-g(x-s,\langle G,e^{-u_{t-s} f}\rangle)\big]ds,
		\end{align}
		which is equivalent to the evolution equation (\ref{2.4}).
	\end{proposition}
	
	\begin{proof}
		We first assume $f \in C_{0}^{1}(0, \infty)^{+}$. In this case, it is easy to see that $\phi(\cdot, f)\in C_{0}^{1}(0,\infty)$. Then $t\mapsto \int_{0}^{t} \phi(x-t+s, U_{s}f)ds$ is belong to $C_{0}^{1}(0,\infty)$. Then $t \mapsto U_{t} f(x)$ is continuously differentiable and solves the differential evolution equation
		\begin{align*}
			\frac{\partial}{\partial t} U_{t} f(x)=-\frac{\partial}{\partial x} U_{t} f(x)-\phi(x, U_{t} f), \quad U_{0} f(x)=1-e^{-f(x)}.
		\end{align*}
		Then $t \mapsto u_{t} f(x)$ is also continuously differentiable. By differentiating both sides of (\ref{2.4}) we have
		\begin{align*}
			e^{-u_{t} f(x)}\frac{\partial}{\partial t} u_{t} f(x)=-e^{-u_{t} f(x)}\frac{\partial}{\partial x} u_{t} f(x)+\alpha(x)e^{-u_{t} f(x)}\big[1-g(x,\langle G, e^{-u_{t}f}\rangle)\big],
		\end{align*}
		it follows that for any $x>0,\ t \ge 0$ and $f \in C_{0}^{1}(0, \infty)^{+}$ we have
		\begin{align}\label{2.7}
			\frac{\partial}{\partial t} u_{t} f(x)=-\frac{\partial}{\partial x} u_{t} f(x)+\alpha(x)\big[1-g(x,\langle G, e^{-u_{t}f}\rangle)\big].
		\end{align}
		Then for any fixed $t\ge 0$ and $x\in (0,\infty)$ we have
		\begin{small}
			\begin{align*}
				\frac{d}{ds} u_sf(x-&t+s)
				=\frac{\partial}{\partial s}u_{s}f(y)\Big |_{y=x-t+s}+\frac{\partial}{\partial s}u_{y}f(x-t+s)\Big|_{y=s}\\
				=&-\frac{\partial}{\partial y}u_{s}f(y)\Big|_{y=x-t+s}+\alpha(y)\big[1-g(y,\langle G, e^{-u_{s}f}\rangle)\big]\Big|_{y=x-t+s}+\frac{\partial}{\partial s}u_{y}f(x-t+s)\Big|_{y=s}\\
				=&\alpha(x-t+s)\big[1-g(x-t+s,\langle G, e^{-u_{s}f}\rangle)\big],\quad 0\le s\le t.
			\end{align*}
		\end{small}By integrating both sides of the above equation over $[0, t]$ we get (\ref{2.7'}) for $f \in C_{0}^{1}(0, \infty)^{+}$. To complete the proof, we only need to show that (\ref{2.7'}) also holds for any
		\begin{align}\label{2.7''}
			f(x)= \begin{cases}
				b_1 & \text{ if } x\in (0,x_1],\\
				b_2 & \text { if } x \in (x_1, x_2],\\
				b_3 & \text { if } x\in (x_2,\infty),
			\end{cases}
		\end{align}
		where $b_2\in [0,\infty)$, $b_1, b_3\in [0,b_2]$ and $x_1, x_2 \in (0,\infty)$. Indeed, let $x_3=(x_2+\frac{1}{n})\vee n$ and
		\begin{align}\label{2.7'''}
			f_{n}(x)= \begin{cases}
				b_{1}e^{1-\frac{1}{1-(nx-1)^2}} & \text { if } x \in(0, \frac{1}{n}),\\
				b_1 & \text { if } x\in [\frac{1}{n},x_1],\\
				b_1+(b_2-b_1)e^{1-\frac{1}{1-(n(x-x_1)-1)^2}} & \text { if } x \in(x_1, x_1 +\frac{1}{n}),\\
				b_2 & \text { if } x\in [x_1+\frac{1}{n},x_2],\\
				b_3+(b_2-b_3)e^{1-\frac{1}{1-(x_2-x)^2 n^2}} & \text { if } x\in (x_2,x_2+\frac{1}{n}],\\
				b_3 & \text { if } x\in (x_2+\frac{1}{n},x_3],\\
				b_{3}e^{1-\frac{1}{1-(x_3-x)^2 n^2}} & \text { if } x\in (x_3,x_3+\frac{1}{n}],\\
				0 & \text { if } x\in (x_3+\frac{1}{n},\infty),
			\end{cases}
		\end{align}
		for $x\in (0, \infty)$ and $n\ge 1$. According to the definition, we have $f_{n}\in C_{0}^{1}(0, \infty)^{+}$ for any $n\ge 1$. It is not difficult to see that $f_n \rightarrow f$ as $n \rightarrow \infty$. Notice that for any $t\ge 0$, by letting $\sigma=\delta_x$ in (\ref{2.5}) and using dominated convergence we have $u_{t}f_n \rightarrow u_{t}f$ as $n \rightarrow \infty$. Thus (\ref{2.7'}) holds for any $f$ given as (\ref{2.7''}) by dominated convergence. Naturally, (\ref{2.7'}) also holds for any simple function $f \in B(0, \infty)^{+}$ in similar ways. The extension to $f \in B(0, \infty)^{+}$ is immediate by monotone convergence. The uniqueness of the solution to (\ref{2.7'}) follows by Gronwall's inequality. Then equations (\ref{2.4}) and (\ref{2.7'}) are equivalent.
		$\hfill\square$
	\end{proof}
	
	\begin{proposition}\label{prop2.2}
		For any $t \ge 0$ and $\sigma \in \mathfrak{N}(0, \infty)$ we have
		\begin{align}\label{2.8}
			\int_{\mathfrak{N}(0,\infty)}\langle\nu, f\rangle Q_{t}(\sigma, d \nu)=\langle\sigma, \pi_{t} f\rangle, \quad f \in B(0,\infty),
		\end{align}
		where $(\pi_{t})_{t \ge 0}$ is the semigroup of bounded kernels on $(0,\infty)$ defined by
		\begin{align}\label{2.9}
			\pi_{t} f(x)=f(x-t)+\int_{0}^{t} \alpha(x-s) g'(x-s,1-) \langle G,\pi_{t-s}f\rangle ds.
		\end{align}
	\end{proposition}
	
	\begin{proof}
		The existence and uniqueness of the locally bounded solution to (\ref{2.9}) follows by a general result; see, e.g., Li \cite[Lemma 2.17]{Li22}. For $f \in B(0,\infty)^{+}$ we can use (\ref{2.7'}) to see that
		\begin{small}
			\begin{align*}
				\frac{\partial}{\partial \theta} u_{t}(\theta f)(x)
				=f(x-t)+\int_{0}^{t} \alpha(x-s) g'(x-s,\langle G,e^{-u_{t-s}(\theta f)}\rangle) \big\langle G, \frac{\partial}{\partial \theta} u_{t-s}(\theta f)\cdot e^{-u_{t-s}(\theta f)}\big\rangle ds.
			\end{align*}
		\end{small}By letting $\theta\rightarrow 0$ we see that the unique solution of (\ref{2.9}) is given by $\pi_{t} f(x)=\frac{\partial}{\partial \theta} u_{t}(\theta f)(x)\big|_{\theta=0}$. By differentiating both sides of (\ref{2.5}) we get (\ref{2.8}). The extension to $f \in B(0,\infty)$ is immediate by linearity.
		$\hfill\square$
	\end{proof}

	\begin{proposition}\label{prop2.3} Let $\|\alpha\|=\sup _{y \ge 0} \alpha(y)$. Then for any $t \ge 0$ and $x>0$ we have
		\begin{align}\label{2.10}
			\pi_{t} f(x) \ge u_{t} f(x) \ge(1-e^{-f(x-t)}) e^{-\|\alpha\| t}, \quad f \in B(0, \infty)^{+}.
		\end{align}
	\end{proposition}
	
	\begin{proof} By taking $\sigma=\delta_{x}$ in both (\ref{2.5}) and (\ref{2.8}), with Jensen's inequality we get the first inequality in (\ref{2.10}). Let $U_{t} f(x)$ be the solution of (\ref{2.6}). By (\ref{2.6}) and a comparison theorem we have $U_{t} f(x) \ge \tilde{U}_{t} f(x)$, where $(t, x) \mapsto \tilde{U}_{t} f(x)$ solves
		\begin{align*}
			\tilde{U}_{t} f(x)=1-e^{-f(x-t)}-\int_{0}^{t} \alpha(x-s) \tilde{U}_{t-s} f(x-s) d s.
		\end{align*}
		The unique locally bounded solution to the above equation is given by
		\begin{align*}
			\tilde{U}_{t} f(x)=(1-e^{-f(x-t)}) \exp \Big\{\!-\!\int_{0}^{t} \alpha(x-s) d s\Big\}.
		\end{align*}
		Then we have the estimate (\ref{2.10}). $\hfill\square$ \end{proof}
	
	\begin{proposition}\label{prop3.11} 		
		For any $x>0$, $t\ge 0$ and $f\in B(0,\infty)^{+}$, there exists a Borel right process $\tilde{\xi}=(\Omega, \mathscr{F}, \mathscr{F}_{t}, \tilde{\xi}_t, \mathbf{P}_{x})$ in $(0,\infty)$ so that
		\begin{align}\label{3.19}
			\pi_{t}f(x)=\mathbf{P}_x \Big[f(\tilde{\xi}_t) \exp\Big\{\!\int_{0}^{t}\alpha(\tilde{\xi}_s)g'(\tilde{\xi}_s,1-)ds\Big\}\Big].
		\end{align}
	\end{proposition}
	
	\begin{proof}
		By (\ref{2.9}) we have
		\begin{align*}
			\|\pi_t f\| \le \|f\|+\|\alpha g'(\cdot,1-)\| \int_0^t\|\pi_{t-s} f\| d s,
		\end{align*}
		where $\|\pi_t f\|=\sup _{x \in (0,\infty)} \pi_t f(x)$. Then by Gronwall's inequality we have $\|\pi_t f\| \le e^{\|\alpha g'(\cdot,1-)\| t}$. On the other hand, for simplicity we let $\tilde{m}(y)=\alpha(y) g'(y,1-)$. For any $t\ge 0$ let $P_{t}^{m}$ be the operator on $B(0,\infty)^{+}$ defined by
		\begin{align*}
			P_t^m f(x)=\exp \Big\{\!-\int_0^t \tilde{m}(x-s) d s\Big\} f(x-t).
		\end{align*}
		By Li \cite[Proposition 2.9]{Li22}, (\ref{2.9}) implies that
		\begin{align}\label{3.16}
			\pi_t f(x)=P_t^m f(x)+\int_0^t P_s^m \tilde{m}(x) \langle G,\pi_{t-s} f\rangle ds+\int_0^t P_s^m(\tilde{m} \pi_{t-s} f)(x) d s.
		\end{align}
		Furthermore, by Li \cite[Theorem A.43]{Li22}, we can define a Borel right semigroup $(\tilde{P}_t)_{t \ge 0}$ by
		\begin{align}\label{3.17}
			\tilde{P}_t f(x)=P_t^m f(x)+\int_0^t P_s^m \tilde{m}(x) \langle G,\tilde{P}_{t-s} f\rangle d s.
		\end{align}
		Iterating (\ref{3.17}) we get
		\begin{small}
			\begin{align*}
				\tilde{P}_t f(x) &=P_t^m f(x)+\sum_{i=1}^{\infty} \int_0^t P_{t-s_1}^m \tilde{m}(x) d s_1 \int_0^{s_1} \langle G,P_{s_1-s_2}^m \tilde{m}\rangle d s_2 \cdots \int_0^{s_{i-1}} \langle G,P_{s_{i-1}-s_i}^m \tilde{m}\rangle \langle G,P_{s_i}^m f\rangle d s_i \\
				&\le
				c \sum_{n=0}^{\infty}\|\tilde{m}\|^n \frac{t^n}{n !}= c \exp \{\|\tilde{m}\| t\}<\infty,
			\end{align*}
		\end{small}where $\|\tilde{m}\|=\sup _{y \in (0,\infty)} \tilde{m}(y)$. Combining (\ref{3.16}) and (\ref{3.17}) gives
		\begin{align*}
			\pi_t f(x)= \tilde{P}_t f(x)+\int_0^t P_s^m(\tilde{m} \pi_{t-s} f)(x) d s+\int_0^t P_s^m \tilde{m}(x)\langle G,\pi_{t-s} f-\tilde{P}_{t-s} f \rangle ds.
		\end{align*}
		Iterating this equality a finite number of times yields
		\begin{align*}
			\pi_t f(x)=
			& \tilde{P}_t f(x)+\int_0^t P_{t-s}^m(\tilde{m} \pi_s f)(x) d s\\
			&+\sum_{i=1}^n \int_0^t P_{t-s_1}^m \tilde{m}(x) d s_1 \int_0^{s_1} \langle G,P_{s_1-s_2}^m \tilde{m}\rangle d s_{2} \cdots
			\int_0^{s_{i-1}} \langle G,P_{s_{i-1}-s_i}^m(\tilde{m} \pi_{s_i} f)\rangle d s_i \\
			&+\Delta_n(t, x),
		\end{align*}
		where\begin{small}
			\begin{align*}
				\Delta_n(t, x)= \int_0^t P_{t-s_1}^m \tilde{m}(x) d s_1 \int_0^{s_1} \langle G,P_{s_1-s_2}^m \tilde{m}\rangle ds_2 \cdots \int_0^{s_{n-1}} \langle G,P_{s_{n-1}-s_n}^m \tilde{m}\rangle\langle G,\pi_{s_n} f-\tilde{P}_{s_n} f\rangle d s_n .
			\end{align*}
		\end{small}Notice that $\|\Delta_n(t, \cdot)\| \le(1+e^{c t})\|f\|\|\tilde{m}\|^n \frac{t^n}{n !}$. By letting $n\rightarrow\infty$ in the above equation we obtain
		\begin{align*}
			\pi_t f(x)=\tilde{P}_t f(x)-\int_0^t \tilde{P}_{t-s}(\tilde{m} \pi_s f)(x) d s.
		\end{align*}
		Let $\tilde{\xi}=(\Omega, \mathscr{F}, \mathscr{F}_{t}, \tilde{\xi}_t, \mathbf{P}_{x})$ be a Borel right realization of the semigroup $(\tilde{P}_t)_{t \ge 0}$. Then by Feynman-Kac formula we have
		\begin{align*}
			\pi_{t}f(x)=\mathbf{P}_x \Big[f(\tilde{\xi}_t) \exp\Big\{\!\int_{0}^{t}\alpha(\tilde{\xi}_s)g'(\tilde{\xi}_s,1-)ds\Big\}\Big],
		\end{align*}
		which completes the proof. $\hfill\square$ \end{proof}

	\begin{proposition}\label{prop2.2'} Suppose that $\|g''(\cdot,1-)\|<\infty$. Then for any $t \ge 0$ and $\sigma \in \mathfrak{N}(0, \infty)$ we have
		\begin{align}\label{2.8'}
			\int_{\mathfrak{N}(0,\infty)}\langle\nu, f\rangle^{2} Q_{t}(\sigma, d \nu)=\langle\sigma, \pi_{t} f\rangle^{2}+\langle\sigma, \gamma_{t} f\rangle, \quad f \in B(0,\infty),
		\end{align}
		where $(\pi_{t})_{t \ge 0}$ is defined by (\ref{2.9}) and $(t,x)\mapsto \gamma_{t}f(x)$ is the unique solution of
		\begin{align}\label{2.9'}
			\gamma_{t} f(x)
			=&\int_{0}^{t} \alpha(x-s) \big[g''(x-s,1-) \langle G,\pi_{t-s}f\rangle^{2} + g'(x-s,1-) \langle G,(\pi_{t-s}f)^{2}\rangle\big] ds \nonumber\\
			&+\int_{0}^{t} \alpha(x-s) g'(x-s,1-) \langle G,\gamma_{t-s}f\rangle ds.
		\end{align}
	\end{proposition}
	
	\begin{proof} This is similar to the proof of Proposition \ref{prop2.2}. We first consider $f \in B(0,\infty)^{+}$. Let
		\begin{align*}
			\gamma_{t}f(x)=-\frac{\partial^2}{\partial \theta^2} u_{t}(\theta f)(x)\Big|_{\theta=0}.
		\end{align*}
		In view of (\ref{2.5}) and (\ref{2.7'}), for any $\theta >0$ we have
		\begin{align*}
			\int_{\mathfrak{N}(0,\infty)}\langle\nu, f\rangle^{2} e^{-\theta\langle \nu, f\rangle}Q_{t}(\sigma, d \nu)=\Big[\big\langle\sigma, \frac{\partial}{\partial \theta} u_{t}(\theta f)\big\rangle^{2}-\big\langle\sigma, \frac{\partial^2}{\partial \theta^2} u_{t}(\theta f)\big\rangle\Big] e^{-\langle \sigma, u_t (\theta f)\rangle}.
		\end{align*}
		and
		\begin{align*}
			\frac{\partial^2}{\partial \theta^2} u_{t}(\theta f)(x)
			=&\int_{0}^{t} \alpha(x-s) g'(x-s,\langle G,e^{-u_{t-s}(\theta f)}\rangle) \Big\langle G, \frac{\partial^2}{\partial \theta^2} u_{t-s}(\theta f)\cdot e^{-u_{t-s}(\theta f)}\Big\rangle ds \nonumber\\
			&-\int_{0}^{t} \alpha(x-s) g'(x-s,\langle G,e^{-u_{t-s}(\theta f)}\rangle) \Big\langle G, \big(\frac{\partial}{\partial \theta} u_{t-s}(\theta f)\big)^{2}\cdot e^{-u_{t-s}(\theta f)}\Big\rangle ds \nonumber\\
			&-\int_{0}^{t} \alpha(x-s) g''(x-s,\langle G,e^{-u_{t-s}(\theta f)}\rangle) \Big\langle G,\frac{\partial}{\partial \theta} u_{t}(\theta f) \cdot e^{-u_{t-s}(\theta f)}\Big\rangle^2 ds.
		\end{align*}
		By letting $\theta \rightarrow 0$ in the two equations we obtain (\ref{2.8'}) and (\ref{2.9'}), first for $f \in B(0,\infty)^{+}$ and then for $f \in B(0,\infty)$. The uniqueness of the solution to (\ref{2.9'}) follows easily by Gronwall's inequality. $\hfill\square$ \end{proof}

	\section{Stochastic equations and martingale problems}\label{section3}
	
	In this section, we give a construction of the remaining-lifetime age-structured branching model as a measure-valued process by the pathwise unique solution of a stochastic integral equation driven by a time-space Poisson random measure. In particular, as an application of the stochastic equations, we obtain the martingale problems associated with the remaining-lifetime age-structured branching processes. For $\mu \in \mathfrak{D}(0,\infty)$ and $\alpha \in B(0,\infty)^{+}$ we define
	\begin{align*}
		A_{\alpha}(\mu, y)=\inf \{z \ge 0:\langle\mu, \alpha 1_{(0, z]}\rangle>\langle\mu, \alpha\rangle y\}, \quad 0 \le y \le 1
	\end{align*}
	with $\inf \emptyset=\infty$ by convention. Then $\langle\mu, \alpha\rangle=0$ implies $A_{\alpha}(\mu, y)=\infty$ for all $0 \le y \le 1$. By an elementary result in probability theory, we have:

	\begin{lemma}\label{lem3.1} If $\langle\mu, \alpha\rangle>0$ and if $\xi$ is a random variable with the uniform distribution on $(0,1]$, then $\mathbf{P}\{A_{\alpha}(\mu, \xi) \in d x\}=\langle\mu, \alpha\rangle^{-1} \alpha(x) \mu(d x), x \ge 0$. \end{lemma}

	Suppose that $(\Omega, \mathscr{F}, \mathscr{F}_{t}, \mathbf{P})$ is a filtered probability space satisfying the usual hypotheses. Let $E:=(0,1] \times \mathbb{N} \times (0,\infty)^{\infty}$ and $w=(y,n,z)\in E$. Let $M(d t, d u, d w, d v)$ be an $(\mathscr{F}_{t})$-Poisson random measure on $(0, \infty)^{2} \times E\times(0,1]$ with intensity $d t d u K(dw) d v$, where $K(dw):=d y \pi(d n) G^{\infty}(d z)$, $\pi(d n)$ denotes the counting measure on $\mathbb{N}$,\ $G^{\infty}(d z)$ denotes the infinite product measure on $(0, \infty)^{\infty}$. Given an $\mathscr{F}_{0}$-measurable random function $X_{0} \in \mathfrak{D}(0, \infty)$, for $t \ge 0$ and $x>0$, we consider the stochastic integral equation
	\begin{small}
		\begin{align}\label{3.1}
			X_{t}(x)\!=\! X_{0}(x+t)\!-\! X_{0}(t)\!+\!\!\int_{0}^{t}\!\! \int_{0}^{\langle X_{s-}, \alpha\rangle}\!\!\! \int_{E} \int_{0}^{p(A_{\alpha}( X_{s-},y), n)} \sum_{i=1}^{n} 1_{\{0<z_{i}-(t-s) \le x\}} M(d s, d u, d w, d v).
		\end{align}
	\end{small}Heuristically, the left-hand side $X_{t}(x)$ is the number of particles alive at time $t$ with remaining lifetimes less than $x$. On the right-hand side, the first two terms $X_{0}(x+t)-X_{0}(t)$ count the number of individuals having remaining lifetimes belonging to $(t,x+t]$ at time 0 and thus still living with remaining lifetimes less than $x$ at time $t$. A reproduction of the population occurs at time $s \in[0, t]$ at rate $\langle X_{s-}, \alpha\rangle d s$. In that case, the remaining lifetime of the mother is distributed according to the probability measure $\langle X_{s-}, \alpha\rangle^{-1} \alpha(x) X_{s-}(d x)$ and is realized as $A_{\alpha}(X_{s-}, y)$, where $y \in(0,1]$ is taken according to the uniform distribution by the Poisson random measure. The number of offspring of the individual takes the value $n \in \mathbb{N}$ with probability $p(A_{\alpha}(X_{s-}, y), n)$. Then the $i$-th individual of those offspring chooses its life-length $z_i\in(0, \infty)$ independently of the others according to the probability distribution $G(d z_i)$ and the birth of it makes effect on the number $X_{t}(x)$ if and only if $0<z_{i}-(t-s) \le x$, $i=1,\ldots,n$. The contribution of the reproduction to $X_{t}(x)$ is recorded by the third term.
	
	Let $\zeta_{a}(x)=1_{\{a \le x\}}$ for $a, x \in \mathbb{R}$. Given a function $f$ on $\mathbb{R}$ we define $f \circ \theta_{t}(x)=f(x+t)$ for $x, t \in \mathbb{R}$. Then we may rewrite (\ref{3.1}) equivalently into
	\begin{small}
		\begin{align}\label{3.2}
			X_{t}(x)=&X_{0} \circ \theta_{t}(x)-X_{0} \circ \theta_{t}(0) \nonumber\\
			&+\int_{0}^{t} \int_{0}^{\langle X_{s-}, \alpha\rangle} \int_{E} \int_{0}^{p(A_{\alpha} (X_{s-},y), n)} \sum_{i=1}^{n} [\zeta_{z_i} \circ \theta_{t-s}(x)-\zeta_{z_i} \circ \theta_{t-s}(0)]\ M(d s, d u, d w, d v).
		\end{align}
	\end{small}A pathwise solution to (\ref{3.2}) is constructed by unscrambling the equation as follows. Let $\tau_{0}=0$. Given $\tau_{k-1} \ge 0$ and $X_{\tau_{k-1}} \in \mathfrak{D}(0, \infty)$, we first define
	\begin{small}
		\begin{align}\label{*}
			\tau_{k}=\tau_{k-1}+\inf \big\{t>0: X_{\tau_{k-1}}(t-\tau_{k-1})+M\big((\tau_{k-1}, \tau_{k-1}+t] \times(0,\langle X_{\tau_{k-1}}, \alpha\rangle] \times H_{k}\big)>0\big\},
		\end{align}
	\end{small}where $H_{k}=\{(w, v): w \in E$, $0<v \le p(A_{\alpha}(X_{\tau_{k-1}}, y), n)\}$, and
	\begin{align}\label{3.3}
		X_{t}(x)=X_{\tau_{k-1}} \circ \theta_{t-\tau_{k-1}}(x), \quad \tau_{k-1} \le t<\tau_{k}, x>0.
	\end{align}
	Then we define
	\begin{align}\label{3.4}
		X_{\tau_{k}}(x)=X_{\tau_{k}-}(x)-X_{\tau_{k}-}(0)+\sum_{i=1}^{n_k}\zeta_{z_i}(x), \quad x>0,
	\end{align}
	where $X_{\tau_{k}-}(x)=X_{\tau_{k-1}} \circ \theta_{\tau_{k}-\tau_{k-1}}(x)$ and $(u_{k}, w_{k}, v_{k}) \in(0, \infty) \times E \times(0,1]$ is the point so that $(\tau_{k}, u_{k}, w_{k}, v_{k}) \in \operatorname{supp}(M)$, where $w_{k}=(y_{k}, n_{k}, z_{k})$. Since $A_{\alpha}(X_{\tau_{k}-}, y_{k}) \in \operatorname{supp}(X_{\tau_{k}-})$, we have $X_{\tau_{k}} \in \mathfrak{D}(0,\infty)$. The expression (\ref{3.4}) means that at time $\tau_{k}$ at least one of the following two events occurs : an individual having remaining lifetime less than $\tau_{k}-\tau_{k-1}$ at time $\tau_{k-1}$ dies or an individual gives birth to $n_{k}$ offspring with life-lengths $(z_{1},\ldots,z_{n_k})\in \mathbb{R}^{n_k}$.
	
	It is clear that (\ref{3.3}) and (\ref{3.4}) uniquely determine the behavior of the trajectory $t \mapsto X_{t}$ on the time intervals $[\tau_{k-1}, \tau_{k}]$, $k=1,2, \cdots$. Let $\tau=\lim _{k \rightarrow \infty} \tau_{k}$ and let $X_{t}=\infty$ for $t \ge \tau$. Then $(X_{t}: t \ge 0)$ is the pathwise unique solution to (\ref{3.2}) up to the lifetime $\tau$. More precisely, (\ref{3.1}) and (\ref{3.2}) hold a.s. with $t$ replaced by $t \wedge \tau_{k}$ for every $k=1,2, \cdots$. Let $m(t)=\sup \{k \ge 0: \tau_{k} \le t\}$ for $t \ge 0$. Let $\beta=\|\alpha g'(\cdot,1-)\|<\infty$. Now we show that $\mathbf{P}(\tau=\infty)=1$.
	
	\begin{lemma}\label{lem3.2}
		Suppose that $\mathbf{E}[X_{0}(\infty)]<\infty$. Then for any $k \ge 1$ we have
		\begin{align}\label{3.5}
			\mathbf{E}\Big[\sup _{0 \le s \le t \wedge \tau_{k}} X_{s}(\infty)\Big] \le \mathbf{E}[X_{0}(\infty)] e^{\beta t}, \quad t \ge 0.
		\end{align}
	\end{lemma}
	
	\begin{proof}
		Recall that $X_{t}(\infty)=\lim _{x \rightarrow \infty} X_{t}(x)=X_{t}((0, \infty))$. Let $\eta_{i}=\inf \{t \ge 0: X_{t}(\infty) \ge i\}$. It is clear that $\lim _{i \rightarrow \infty} \eta_{i}=\tau$. Let $\zeta_{i, k}=\eta_{i} \wedge \tau_{k}$. In view of (\ref{3.2}), we have
		\begin{small}
			\begin{align}\label{3.6}
				X_{t}(\infty)\!=\!X_{0}(\infty)\!-\!X_{0}(t)\!+\!\!\int_{0}^{t}\!\! \int_{0}^{\langle X_{s-}, \alpha\rangle}\!\!\! \int_{E}\int_{0}^{p(A_{\alpha}(X_{s-}, y), n)}\! n-\sum_{i=1}^{n} 1_{\{z_i -(t-s)\le 0\}} M(d s, d u, d w, d v).
			\end{align}
		\end{small}It follows that
		\begin{align*}
			\mathbf{E}\Big[\sup _{0 \le s \le t \wedge \zeta_{i, k}} X_{s}(\infty)\Big]
			&\le \mathbf{E}[X_{0}(\infty)]+\sum_{n \in \mathbb{N}} \mathbf{E}\Big[\int_{0}^{t \wedge \zeta_{i, k}}\langle X_{s-}, \alpha\rangle d s \int_{0}^{1} p(A_{\alpha}(X_{s-}, y), n) n d y\Big] \\
			&=\mathbf{E}[X_{0}(\infty)]+\sum_{n \in \mathbb{N}} \mathbf{E}\Big[\int_{0}^{t \wedge \zeta_{i, k}} d s \int_{0}^{\infty} \alpha(y) p(y, n) n d X_{s-}(y)\Big] \\
			&=\mathbf{E}[X_{0}(\infty)]+\mathbf{E}\Big[\int_{0}^{t \wedge \zeta_{i, k}} d s \int_{0}^{\infty} \alpha(y) g'(y,1-) d X_{s-}(y)\Big] \\
			&\le \mathbf{E}[X_{0}(\infty)]+\beta\ \mathbf{E}\Big[\int_{0}^{t \wedge \zeta_{i, k}} X_{s-}(\infty) d s\Big].
		\end{align*}
		Clearly, we have $X_{s-}(\infty) \le i$ for $0<s \le t \wedge \zeta_{i, k}$. Then $\mathbf{E}\big[\sup _{0 \le s \le t \wedge \zeta_{i, k}} X_{s}(\infty)\big]$ is locally bounded in $t \ge 0$. Since the trajectory $s \mapsto X_{s}(\infty)$ has at most countably many jumps, it follows that
		\begin{align*}
			\mathbf{E}\Big[\sup _{0 \le s \le t \wedge \zeta_{i, k}} X_{s}(\infty)\Big]
			&\le \mathbf{E}[X_{0}(\infty)]+\beta\ \mathbf{E}\Big[\int_{0}^{t \wedge \zeta_{i, k}} X_{s}(\infty) d s\Big] \\
			&\le \mathbf{E}[X_{0}(\infty)]+\beta \int_{0}^{t} \mathbf{E}\Big[X_{s \wedge \zeta_{i, k}}(\infty)\Big] d s \\
			&\le \mathbf{E}[X_{0}(\infty)]+\beta \int_{0}^{t} \mathbf{E}\Big[\sup _{0 \le r \le s \wedge \zeta_{i, k}} X_{r}(\infty)\Big] d s .
		\end{align*}
		By Gronwall's inequality we have
		\begin{align*}
			\mathbf{E}\Big[\sup _{0 \le s \le t \wedge \zeta_{i, k}} X_{s}(\infty)\Big] \le \mathbf{E}[X_{0}(\infty)] e^{\beta t}.
		\end{align*}
		Then letting $i \rightarrow \infty$ and using Fatou's lemma we obtain (\ref{3.5}).
		$\hfill\square$
	\end{proof}
	
	\begin{proposition}\label{prop3.3}
		Suppose that $\mathbf{E}[X_{0}(\infty)]<\infty$. Then we have $\mathbf{P}(\tau=\infty)=1$ and
		\begin{align}\label{3.7}
			\mathbf{E}[m(t)] \le \mathbf{E}[X_{0}(\infty)]\Big[1+\|\alpha\| \int_{0}^{t} e^{\beta s} d s\Big], \quad t \ge 0.
		\end{align}
	\end{proposition}
	
	\begin{proof}
		By (\ref{3.2}) and monotone convergence we have
		\begin{align*}
			\mathbf{E}[m(t)] &=\lim _{k \rightarrow \infty} \mathbf{E}\Big[X_{0}(t\wedge \tau_{k})+\int_{0}^{t \wedge \tau_{k}} \int_{0}^{\langle X_{s-}, \alpha\rangle} \int_{E} \int_{0}^{p(A_{\alpha}(X_{s-}, y), n)} M(d s, d u, d w, d v)\Big] \\
			&=\lim _{k \rightarrow \infty} \mathbf{E}\Big[X_{0}(t\wedge \tau_{k})+\sum_{n \in \mathbb{N}} \int_{0}^{t \wedge \tau_{k}}\langle X_{s-}, \alpha\rangle d s \int_{0}^{1} p(A_{\alpha}(X_{s-}, y), n) d y\Big] \\
			&=\lim _{k \rightarrow \infty} \mathbf{E}\Big[X_{0}(t\wedge \tau_{k})+\sum_{n \in \mathbb{N}} \int_{0}^{t \wedge \tau_{k}} \alpha(y) d s \int_{0}^{\infty} p(y, n) d X_{s-}(y)\Big] \\
			&\le \lim _{k \rightarrow \infty}\Big\{\mathbf{E}[X_{0}(t\wedge \tau_{k})]+\|\alpha\| \mathbf{E}\Big[\int_{0}^{t \wedge \tau_{k}} X_{s}(\infty) d s\Big]\Big\}\\
			&=\lim _{k \rightarrow \infty}\Big\{\mathbf{E}[X_{0}(t\wedge \tau_{k})]+\|\alpha\| \int_{0}^{t} \mathbf{E}[X_{s \wedge \tau_{k}}(\infty)] d s\Big\}.
		\end{align*}
		Then (\ref{3.7}) follows by (\ref{3.5}) and monotone convergence. In particular, we have $\mathbf{P}\{\tau>t\}=\mathbf{P}\{m(t)<\infty\}=1$ for every $t \ge 0$. That implies $\mathbf{P}\{\tau=\infty\}=1$.
		$\hfill\square$
	\end{proof}
	
	\begin{proposition}\label{prop3.4}
		Suppose that $\mathbf{E}[X_{0}(\infty)]<\infty$. Then we have
		\begin{align*}
			\mathbf{E}\Big[\sup _{0 \le s \le t} X_{s}(\infty)\Big] \le \mathbf{E}[X_{0}(\infty)] e^{\beta t}, \quad t \ge 0.
		\end{align*}
	\end{proposition}
	
	\begin{proof}
		Since $\mathbf{P}\{\tau=\infty\}=1$ by Proposition \ref{prop3.3}, we obtain the result from (\ref{3.5}) by using monotone convergence.
		$\hfill\square$
	\end{proof}
	
	\begin{proposition}\label{prop3.9}
		Suppose that $\mathbf{E}[X_{0}(\infty)]<\infty$. Then for any $t \ge 0$ we have
		\begin{align}\label{3.13}
			\mathbf{E}[X_{t}(\infty) ] \le \mathbf{E}[ X_{0}(\infty)] e^{\beta t}-\mathbf{E}[X_{0}(t)]-\beta \int_{0}^{t}e^{\beta (t-s)}\mathbf{E}[X_{0}(s)]ds,
		\end{align}
	\end{proposition}
	\begin{proof}
		The inequality in (\ref{3.13}) follows from the similar argument as Lemma \ref{lem3.2}. Since the trajectory $t \mapsto X_{t}(\infty)$ has at most countably many jumps, by taking expectations in (\ref{3.6}) we obtain \begin{small}
			\begin{align*}
				\mathbf{E}[X_{t}(\infty)]=\mathbf{E}&[X_{0}(\infty)]-\mathbf{E}[X_{0}(t)]\\
				+&\sum_{n \in \mathbb{N}} \mathbf{E}\Big\{\!\int_{0}^{t}\!\langle X_{s-}, \alpha\rangle d s \int_{0}^{1} \!p(A_{\alpha}(X_{s-}, y), n)\Big[n-\!\!\int_{(0, \infty)^{\infty}}\sum_{i=1}^{n}1_{\{z_i-(t-s)\le 0\}}G^{\infty}(dz)\Big]d y\Big\}\\
				=\mathbf{E}&[X_{0}(\infty)]-\mathbf{E}[X_{0}(t)]\\
				+&\sum_{n \in \mathbb{N}} \mathbf{E}\Big\{\!\int_{0}^{t} d s \int_{0}^{\infty} \alpha(y) p(y, n)n\Big[1-\int_{0}^{\infty}1_{\{z-(t-s)\le 0\}}G(dz)\Big] d X_{s}(y)\Big\}\\
				=\mathbf{E}&[X_{0}(\infty)]-\mathbf{E}[X_{0}(t)]+\int_{0}^{t} \mathbf{E}\Big\{\!\int_{0}^{\infty} \alpha(y) g'(y,1-) [1-G(t-s)] d X_{s}(y)\Big\} ds\\
				\le \mathbf{E}&[X_{0}(\infty)]-\mathbf{E}[X_{0}(t)]+\int_{0}^{t} \beta\ \mathbf{E}[X_{s}(\infty)] d s .
			\end{align*}
		\end{small}Then we get (\ref{3.13}) by a comparison result.
		$\hfill\square$
	\end{proof}
	
	By Situ \cite[Theorem 137]{Situ05}, Proposition \ref{prop3.3} and the comments before Proposition \ref{lem3.2}, we naturally get the following theorem and further obtain that the solution of (\ref{3.1}) or (\ref{3.2}) determines a measure-valued strong Markov process $(X_{t}: t \ge 0)$.
	
	\begin{theorem}\label{th3.4}
		Given any initial value $X_{0} \in \mathfrak{D}(0, \infty)$, there exists a pathwise unique strong solution with infinite lifetime to (\ref{3.1}) or (\ref{3.2}).
	\end{theorem}
	
	The following propositions give some useful characterizations of the process identified by the stochastic equation (\ref{3.1}) or (\ref{3.2}).
	
	\begin{proposition}\label{prop3.5}
		For any $t \ge 0$ and $f \in B(0, \infty)$ we have
		\begin{small}
			\begin{align}\label{3.8}
				\langle X_{t}, f\rangle=
				&\langle X_{0}, f \circ \theta_{-t}\rangle \nonumber\\
				&+\int_{0}^{t} \int_{0}^{\langle X_{s-}, \alpha\rangle}\int_{E} \int_{0}^{p(A_{\alpha}(X_{s-},y),n)} \sum_{i=1}^{n} f(z_{i}-(t-s))1_{\{z_{i}-(t-s)>0\}} M(d s, d u, d w, d v) .
			\end{align}
		\end{small}
	\end{proposition}
	
	\begin{proof}
		For any fixed integer $m \ge 1$, let $x_{j}=j m / 2^{m}$ with $j=0,1, \cdots, 2^{m} .$ By (\ref{3.2}) it holds almost surely for any $f \in C_{0}^{1}(0, \infty)$ that
		\begin{small}
			\begin{align*}
				\sum_{j=1}^{2^{m}} f^{\prime}(x_{j}) &X_{t}(x_{j})-\!\sum_{j=1}^{2^{m}} f^{\prime}(x_{j}) X_{0} \circ \theta_{t}(x_{j})+\!\sum_{j=1}^{2^{m}} f^{\prime}(x_{j}) X_{0} \circ \theta_{t}(0) \\
				=\int_{0}^{t}\! &\int_{0}^{\langle X_{s-}, \alpha\rangle}\!\!\! \int_{E}\! \int_{0}^{p(A_{\alpha}(X_{s-}, y), n)} \sum_{j=1}^{2^{m}} f^{\prime}(x_{j})\sum_{i=1}^{n}[\zeta_{z_i} \circ \theta_{t-s}(x_{j})-\zeta_{z_i} \circ \theta_{t-s}(0)] M(d s, d u, d w, d v).
			\end{align*}
		\end{small}Then we multiply the above equation by $2^{-m}$ and let $m \rightarrow \infty$ to see, almost surely,
		\begin{small}
			\begin{align}\label{3.9}
				&\int_{0}^{\infty}f'(x)X_{t}(x) dx -\int_{0}^{\infty}f'(x)X_{0}\circ \theta_{t}(x)dx+\int_{0}^{\infty}f'(x)X_{0}\circ\theta_{t}(0)dx \nonumber\\
				&\!\!=\!\!\int_{0}^{t}\!\! \int_{0}^{\langle X_{s-}, \alpha\rangle}\!\!\!\! \int_{E}\! \int_{0}^{p(A_{\alpha}(X_{s-},y), n)}\!\!\Big\{\!\int_{0}^{\infty}\!\!\! f'(x)\!\sum_{i=1}^{n} [\zeta_{z_i} \circ \theta_{t-s}(x)\!-\!\zeta_{z_i} \circ \theta_{t-s}(0)]dx\Big\} M(d s, d u, d w, d v).
			\end{align}
		\end{small}By integration by parts we get
		\begin{align}\label{3.10}
			\langle X_{t}, f\rangle=-\int_{0}^{\infty} f^{\prime}(x) X_{t}(x) d x.
		\end{align}
		Since $f \in C_{0}^{1}(0, \infty)$, we have
		\begin{align}\label{3.101}
			\int_{0}^{\infty}f'(x)X_{0}\circ\theta_{t}(0)dx=-f(0) X_0 (t)=0.
		\end{align}
		and
		\begin{align}\label{3.102}
			\int_{0}^{\infty}f'(x)\sum_{i=1}^{n} [\zeta_{z_i} \circ \theta_{t-s}(x)-\zeta_{z_i} \circ \theta_{t-s}(0)]dx
			&=\sum_{i=1}^{n} \int_{z_{i}-(t-s)}^{\infty} f'(x) dx 1_{\{z_{i}-(t-s)>0\}} \nonumber\\
			&=-\sum_{i=1}^{n} f(z_i -(t-s)) 1_{\{z_{i}-(t-s)>0\}}.
		\end{align}
		From (\ref{3.9}), (\ref{3.10}), (\ref{3.101}) and (\ref{3.102}) we see that (\ref{3.8}) holds for any $f \in C_{0}^{1}(0, \infty)$. Then the relation also holds for any $f \in B(0, \infty)$ by a monotone class argument.
		$\hfill\square$
	\end{proof}
	
	\begin{proposition}\label{prop3.6} For any $t \ge 0$ and $f \in C_{0}^{1}(0, \infty)$ with the first derivative vanishing at $0$, we have
		\begin{small}
			\begin{align}\label{3.11}
				\langle X_{t}, f\rangle=\langle X_{0}, f\rangle-\!\!\int_{0}^{t}\langle X_{s-}, f^{\prime}\rangle d s +\!\!\int_{0}^{t}\! \int_{0}^{\langle X_{s-}, \alpha\rangle}\!\!\! \int_{E}\! \int_{0}^{p(A_{\alpha}(X_{s-}, y), n)}\sum_{i=1}^{n}f(z_i) M(d s, d u, d w, d v) .
			\end{align}
		\end{small}
	\end{proposition}
	
	\begin{proof} For $m \ge 1$ we consider a partition $\Delta_{m}=\{0=t_{0}<t_{1}<\cdots<t_{m}=t\}$ of $[0, t]$. Notice that $\frac{\partial}{\partial t}(f \circ \theta_{t})(x)=f^{\prime}(x+t)$. Since $f(0)=0$, by (\ref{3.8}) we have
		\begin{small}
			\begin{align*}
				\langle X_{t}, f\rangle=&\ \langle X_{0}, f\rangle+\sum_{j=1}^{m}[\langle X_{t_{j}}, f\rangle-\langle X_{t_{j-1}}, f \circ \theta_{t_{j-1}-t_{j}}\rangle]+\sum_{j=1}^{m}[\langle X_{t_{j-1}}, f \circ \theta_{t_{j-1}-t_{j}}\rangle-\langle X_{t_{j-1}}, f\rangle]\\
				=&\ \langle X_{0}, f\rangle - \sum_{j=1}^{m}
				\int_{t_{j-1}}^{t_{j}}\langle X_{t_{j-1}}, f^{\prime} \circ \theta_{t_{j-1}-s}\rangle d s\\
				&\ +\sum_{j=1}^{m}\!
				\int_{t_{j-1}}^{t_{j}}\! \int_{0}^{\langle X_{s-}, \alpha\rangle}\!\!\!\! \int_{E}\! \int_{0}^{p(A_{\alpha}(X_{s-}, y), n)}\!\sum_{i=1}^{n}f(z_i -(t_j -s)) 1_{\{z_i -(t_j -s)>0\}} M(d s, d u, d w, d v).
			\end{align*}
		\end{small}where $\langle X_{t_{j-1}}, f^{\prime} \circ \theta_{t_{j-1}-s}\rangle$ is well-defined since the first derivative of $f$ vanishes at $0$. By letting $|\Delta_{m}|:=\max _{1 \le j \le m}(t_{j}-t_{j-1}) \rightarrow 0$ and using the right continuity of $s \mapsto X_{s}$ and the continuity of $s \mapsto f \circ \theta_{s}$ we obtain (\ref{3.11}). $\hfill\square$ \end{proof}
	
	Let $C^{1}[0, \infty)$ be the set of bounded continuous Borel functions on $[0,\infty)$ with bounded continuous derivatives of the first order. Then we have the following result.

	\begin{proposition}\label{prop3.7} For $f \in C_{0}^{1}(0, \infty)$ with the first derivative vanishing at $0$ and $F \in C^{1}[0, \infty)$, let $F_{f}(\mu)=F(\langle\mu, f\rangle)$ and let
		\begin{small}
			\begin{align*}
				\mathscr{L}_{0} F_{f}(\mu)\!=\!-\langle\mu, f^{\prime}\rangle F^{\prime}\!(\langle\mu, f\rangle)\! -\!\! \sum_{n \in \mathbb{N}} \int_{0}^{\infty}\!\!\! \alpha(y) p(y, n)\!\int_{(0, \infty)^{\infty}}\!\!\Big[F(\langle\mu, f\rangle)\!-\! F\Big(\langle\mu, f\rangle\!+\!\!\sum_{i=1}^{n}f(z_i)\Big)\Big] G^{\infty}(dz)\mu(d y).
			\end{align*}
		\end{small}Then we have
		\begin{align}\label{3.12}
			F_{f}(X_{t})=F_{f}(X_{0})+\int_{0}^{t} \mathscr{L}_{0} F_{f}(X_{s}) d s+\text {mart}.
		\end{align}
	\end{proposition}
	
	\begin{proof} Let $\tilde{M}$ denote the compensated measure of $M$. Since the process $s \mapsto X_{s}$ has at most countably many jumps, by Proposition \ref{prop3.6} and It$\mathrm{\hat{o}}$'s formula, we have
		\begin{small}
			\begin{align*}
				F(\langle X_{t}, f\rangle)= F&(\langle X_{0}, f\rangle)-\int_{0}^{t} F^{\prime}(\langle X_{s}, f\rangle)\langle X_{s}, f^{\prime}\rangle d s \\
				-&\!\int_{0}^{t}\! \int_{0}^{\langle X_{s-}, \alpha\rangle}\!\!\! \int_{E} \!\int_{0}^{p(A_{\alpha}(X_{s-}, y), n)}\!\Big[F(\langle X_{s-}, f\rangle)\!-\!F\Big(\langle X_{s-}, f\rangle\!+\!\sum_{i=1}^{n}f(z_i)\Big)\Big]\! M(d s, d u, d w, d v) \\
				=F&(\langle X_{0}, f\rangle)-\int_{0}^{t} F^{\prime}(\langle X_{s}, f\rangle)\langle X_{s}, f^{\prime}\rangle d s-M_{t}^{F}(f) \\
				-&\!\int_{0}^{t} d s\!\! \int_{0}^{\infty}\!\! \alpha(y) \sum_{n \in \mathbb{N}} p(y, n)\!\!\int_{(0, \infty)^{\infty}}\!\Big[F(\langle X_{s-}, f\rangle)\!-\!F\Big(\langle X_{s-}, f\rangle\!+\!\sum_{i=1}^{n}f(z_i)\Big)\Big]G^{\infty}(dz)X_{s-}(dy)\\
				=F&(\langle X_{0}, f\rangle)+\int_{0}^{t} \mathscr{L}_{0} F_{f}(X_{s}) d s-M_{t}^{F}(f).
			\end{align*}
		\end{small}where
		\begin{small}
			\begin{align*}
				M_{t}^{F}(f)=\!\int_{0}^{t}\! \int_{0}^{\langle X_{s-}, \alpha\rangle}\!\!\! \int_{E} \int_{0}^{p(A_{\alpha}(X_{s-}, y), n)}\!\Big[F(\langle X_{s-}, f\rangle)-F\Big(\langle X_{s-}, f\rangle+\sum_{i=1}^{n}f(z_i)\Big)\Big] \tilde{M}(d s, d u, d w, d v),
			\end{align*}
		\end{small}and $\langle X_{s}, f^{\prime} \rangle$ is well-defined since the first derivative of $f$ vanishes at $0$. By Proposition \ref{prop3.4} one can check that $\{M_{t}^{F}(f): t \ge 0\}$ is a martingale.
		$\hfill\square$
	\end{proof}
	
	\begin{theorem}\label{th3.8}
		The measure-valued process $(X_{t}: t \ge 0)$ defined by the stochastic equation (\ref{3.1}) or (\ref{3.2}) is an $(\alpha,g,G)$-process.
	\end{theorem}
	
	\begin{proof} By Proposition \ref{prop3.7} we can see $(X_{t}: t \ge 0)$ solves the martingale problem (\ref{3.12}). Recall that $(X_{t}: t \ge 0)$ is a c\`{a}dl\`{a}g process. Let $e_{f}(\mu)=e^{-\langle\mu, f\rangle}$ and let $\mathscr{L}_{0}$ be defined as in Proposition \ref{prop3.7}. For $f \in C_{0}^{1}(0, \infty)$ with the first derivative vanishing at $0$, it is elementary to see that
		\begin{align*}
			\mathscr{L}_{0} e_{f}(\mu) &=\langle\mu, f^{\prime}\rangle e^{-\langle\mu, f\rangle}-e^{-\langle\mu, f\rangle} \int_{0}^{\infty}\!\! \alpha(y)\int_{(0, \infty)^{\infty}}\Big[1-\sum_{n \in \mathbb{N}} p(y, n) e^{-\sum_{i=1}^{n}f(z_i)}\Big] G^{\infty}(dz)\mu(d y) \\
			&=\langle\mu, f^{\prime}\rangle e^{-\langle\mu, f\rangle}-e^{-\langle\mu, f\rangle} \int_{0}^{\infty}\!\! \alpha(y)\Big[1-\sum_{n \in \mathbb{N}} p(y, n) \int_{(0, \infty)^{\infty}} e^{-\sum_{i=1}^{n}f(z_i)} G^{\infty}(dz)\Big]\mu(d y) \\
			&=e^{-\langle\mu, f\rangle}\langle\mu, f^{\prime}\rangle-e^{-\langle\mu, f\rangle} \int_{0}^{\infty}\!\! \alpha(y)[1-g(y, \langle G, e^{-f}\rangle)] \mu(d y).
		\end{align*}
		Let $\mathscr{F}_{t}^{X}=\sigma\{X_{s}: 0 \le s \le t\}$. By (\ref{2.7}), (\ref{3.12}) and the mean-value theorem, we have
		\begin{align*}
			e^{-\langle X_{t}, u_{T-t} f\rangle}= e&^{-\langle X_{0}, u_{T} f\rangle}+\sum_{i=0}^{\infty}\big[e^{-\langle X_{t \wedge(i+1) / k}, u_{T-t \wedge i / k} f\rangle}-e^{-\langle X_{t \wedge i / k}, u_{T-t \wedge i / k} f\rangle}\big] \\
			+&\sum_{i=0}^{\infty}\big[e^{-\langle X_{t \wedge(i+1) / k}, u_{T-t \wedge(i+1) / k} f\rangle}-e^{-\langle X_{t \wedge(i+1) / k}, u_{T-t \wedge i / k} f\rangle}\big] \\
			=e&^{-\langle X_{0}, u_{T} f\rangle}+\sum_{i=0}^{\infty} \int_{t \wedge i / k}^{t \wedge(i+1) / k} e^{-\langle X_{s}, u_{T-t \wedge i / k} f\rangle}\Big\{\big\langle X_{s}, \frac{\partial}{\partial x} u_{T-t \wedge i / k} f\rangle\\
			-&\int_{0}^{\infty} \alpha(y)[1- g(y, \langle G, e^{-u_{T-t \wedge i / k}f}\rangle)] X_{s}(d y)\Big\} d s \\
			+&M_{k}(t)+\sum_{i=0}^{\infty} \int_{t \wedge i / k}^{t \wedge(i+1) / k} e^{-\xi_{k}(t)}\Big\{\big\langle X_{t \wedge(i+1) / k}, -\frac{\partial}{\partial x} u_{T-s} f\big\rangle\\
			+&\int_{0}^{\infty} \alpha(y)[1- g(y, \langle G, e^{-u_{T-s}f}\rangle)] X_{t \wedge(i+1) / k}(d y)\Big\} d s,
		\end{align*}
		where $t \mapsto M_{k}(t)$ is an $(\mathscr{F}_{t}^{X})$-martingale and
		\begin{align*}
			\langle X_{t \wedge(i+1) / k}, u_{T-t \wedge i / k} f\rangle \le \xi_{k}(t) \le\langle X_{t \wedge(i+1) / k}, u_{T-t \wedge(i+1) / k} f\rangle
		\end{align*}
		or
		\begin{align*}
			\langle X_{t \wedge(i+1) / k}, u_{T-t \wedge(i+1) / k} f\rangle \le \xi_{k}(t) \le\langle X_{t \wedge(i+1) / k}, u_{T-t \wedge i / k} f\rangle.
		\end{align*}
		By letting $k \rightarrow \infty$ we see that $t \mapsto e^{-\langle X_{t}, u_{T-t} f\rangle}$ is an $(\mathscr{F}_{t}^{X})$-martingale. In particular, for any $f \in C_{0}^{1}(0, \infty)$ with the first derivative vanishing at $0$, we have
		\begin{align}\label{3.81}
			\mathbf{E}\big[e^{-\langle X_{T}, f\rangle}\big|\mathscr{F}_{t}^{X}\big]=e^{-\langle X_{t}, u_{T-t} f\rangle}, \quad T \ge t \ge 0.
		\end{align}
		The extension to $f \in B(0, \infty)^{+}$ also follows by arguments similar to those in the proof of Proposition \ref{prop2.1}.
		$\hfill\square$
	\end{proof}

	It follows from Theorem \ref{th3.8} that the solution of the martingale problem (\ref{3.12}) is an $(\alpha,g,G)$-process. We further easily infer that the uniqueness of the solution to the martingale problem (\ref{3.12}) holds. Then as a corollary of Theorem \ref{th3.8} we get the following:
	
	\begin{corollary}\label{cor3}
		The martingale problem (\ref{3.12}) is well-posed.
	\end{corollary}
	
	Actually, the characterizations of the $(\alpha,g,G)$-processes in terms of Laplace transforms (\ref{2.5}), stochastic equations (\ref{3.1}) (or (\ref{3.2})) or martingale problems (\ref{3.12}) are equivalent in the sense of distribution. An interesting open problem is starting from a measure-valued process with transition semigroup given by (\ref{2.5}), how one can construct the noise so that it satisfies the stochastic equation (\ref{3.1}) or (\ref{3.2}).

	\section{The branching process with immigration}\label{setion4}
	
	In this section we introduce an remaining-lifetime age-structured branching process with immigration and discuss its ergodic property. Let $\psi$ be a functional on $B(0,\infty)^{+}$ given by
	\begin{align*}
		\psi(f)=\int_{\mathfrak{N}(0,\infty)^{\circ}}(1-e^{-\langle \nu, f\rangle}) L(d \nu),\quad f \in B(0,\infty)^{+},
	\end{align*}
	where $L(d\nu)$ is a finite measure on $\mathfrak{N}(0,\infty)^{\circ}:=\mathfrak{N}(0,\infty) \backslash \mathbf{0}$, and $\mathbf{0}$ denotes the null measure. A Markov process $Y = (Y_t : t \ge 0)$ with state space $\mathfrak{N}(0,\infty)$ is called a \textit{remaining-lifetime age-structured branching process with immigration with parameters $(\alpha,g,G,\psi)$} or simply an \textit{$(\alpha,g,G,\psi)$-process} if it has the transition semigroup $(P_t)_{t \ge 0}$ defined by
	\begin{align}\label{4.1}
		\int_{\mathfrak{N}(0,\infty)} e^{-\langle \nu, f\rangle} P_t(\sigma, d \nu)=\exp \Big\{\!-\langle\sigma, u_t f\rangle-\int_0^t \psi(u_s f) d s\Big\},\quad f \in B(0,\infty)^{+},
	\end{align}
	where $u_{t} f(x)$ is the unique solution to (\ref{2.4}) or (\ref{2.7'}). Such a process is characterized by the properties (i)-(iii) given in Section \ref{section2}, along with the following:
	\begin{description}
		
		\item[\textmd{(iv)}] The immigrants come according to a Poisson random measure on $(0, \infty) \times \mathfrak{N}(0,\infty)^{\circ}$ with intensity $dsL(d\nu)$.
		
	\end{description}

	Suppose that $(\Omega, \mathscr{F}, \mathscr{F}_{t}, \mathbf{P})$ is a filtered probability space satisfying the usual hypotheses. Let $M(dt, du, dw, dv)$ be a Poisson random measure as in Section \ref{section3}, and let $M_{1}(dt, d\nu)$ be an $(\mathscr{F}_t)$-Poisson random measure on $(0, \infty) \times \mathfrak{N}(0,\infty)^{\circ}$ with intensity $dt L(d\nu)$. We assume that the two random measures are independent of each other. For $t \ge 0$ and $x>0$, we consider the stochastic integral equation
	\begin{small}\begin{align}\label{4.2}
			Y_{t}(x)=&\ Y_{0} \circ \theta_{t}(x)-Y_{0} \circ \theta_{t}(0)+\int_{0}^{t}\int_{\mathfrak{N}(0,\infty)^{\circ}}[\nu \circ \theta_{t-s}(x)-\nu \circ \theta_{t-s}(0)]\ M_{1}(ds,d\nu)\nonumber\\
			&
			+\int_{0}^{t} \int_{0}^{\langle Y_{s-}, \alpha\rangle} \int_{E} \int_{0}^{p(A_{\alpha} (Y_{s-},y), n)} \sum_{i=1}^{n} [\zeta_{z_i} \circ \theta_{t-s}(x)-\zeta_{z_i} \circ \theta_{t-s}(0)]\ M(d s, d u, d w, d v).
		\end{align}
	\end{small}Here, the first two terms on the right-hand side are as explained for (\ref{3.1}) and the last term represents the immigration. A pathwise-unique solution to (\ref{4.2}) is constructed as follows. Let $\sigma_0=0$. Given $\sigma_{k-1}\ge 0$, we first define
	\begin{align*}
		\sigma_k=\sigma_{k-1}+\inf \{t>0: M_{1}((\sigma_{k-1}, \sigma_{k-1}+t] \times \mathfrak{N}(0,\infty)^{\circ})>0\}
	\end{align*}
	and
	\begin{align*}
		Y_t(x)=X_{k, t-\sigma_{k-1}}(x),\quad \sigma_{k-1} \le t<\sigma_k,
	\end{align*}
	where $(X_{k, t}(x): t \ge 0)$ is the pathwise-unique solution to the following equation:
	\begin{small}
		\begin{align*}
			X_{t}(x)=&\ Y_{\sigma_{k-1}} \circ \theta_{t}(x)-Y_{\sigma_{k-1}} \circ \theta_{t}(0)\\
			&+\int_{0}^{t} \int_{0}^{\langle X_{s-}, \alpha\rangle} \int_{E} \int_{0}^{p(A_{\alpha} (X_{s-},y), n)} \sum_{i=1}^{n} [\zeta_{z_i} \circ \theta_{t-s}(x)-\zeta_{z_i} \circ \theta_{t-s}(0)]\ M(\sigma_{k-1}+d s, d u, d w, d v).
		\end{align*}
	\end{small}Then we define
	\begin{align*}
		Y_{\sigma_k}(x)=Y_{\sigma_k-}(x)+\int_{\{\sigma_k\}} \int_{\mathfrak{N}(0,\infty)^{\circ}} \nu(x) M_{1}(ds, d\nu) .
	\end{align*}
	It is easy to see that $\lim _{k \rightarrow \infty} \sigma_k=\infty$ and $(Y_t: t \ge 0)$ is the pathwise-unique solution to (\ref{4.2}).
	
	The solution to (\ref{4.2}) determines a measure-valued strong Markov process $(Y_t : t \ge 0)$ with state space $\mathfrak{N}(0,\infty)$. We omit the proofs of some of the following results since the arguments are similar to those for the corresponding results in Section \ref{section3}.
	
	\begin{proposition}\label{prop4.0}
		For any $t \ge 0$ and $\sigma \in \mathfrak{N}(0, \infty)$ we have
		\begin{align}\label{4.0}
			\int_{\mathfrak{N}(0, \infty)}\langle\nu, f\rangle P_{t}(\sigma, d \nu)=\langle\sigma, \pi_{t} f\rangle+\int_{\mathfrak{N}(0,\infty)^{\circ}} L(d \nu) \int_{0}^{t} \langle \nu,\pi_{s}f\rangle ds, \quad f \in B(0, \infty),
		\end{align}
		where $\pi_{t} f(x)$ is the unique solution to (\ref{2.9}).
	\end{proposition}
	
	\begin{proposition}\label{prop4.1}
		For any $t \ge 0$ and $f \in B(0, \infty)^{+}$ we have
		\begin{align*}
			\langle Y_{t}, f\rangle=&\ \langle Y_{0}, f \circ \theta_{-t}\rangle+\int_{0}^{t}\int_{\mathfrak{N}(0,\infty)^{\circ}}\langle\nu,f\circ \theta_{s-t}\rangle\ M_{1}(ds,d\nu)\\
			&+\int_{0}^{t} \int_{0}^{\langle Y_{s-}, \alpha\rangle}\int_{E} \int_{0}^{p(A_{\alpha}(Y_{s-},y), n)}\sum_{i=1}^{n} f(z_{i}-(t-s))1_{\{z_{i}-(t-s)>0\}}\ M(d s, d u, d w, d v).
		\end{align*}
	\end{proposition}
	
	\begin{proposition}\label{prop4.2}
		For any $t \ge 0$ and $f \in C_{0}^{1}(0, \infty)^{+}$ with the first derivative vanishing at $0$, we have
		\begin{align}\label{4.3}
			\langle Y_{t}, f\rangle=
			&\ \langle Y_{0}, f\rangle-\int_{0}^{t}\langle Y_{s-}, f^{\prime}\rangle d s+\int_{0}^{t}\int_{\mathfrak{N}(0,\infty)^{\circ}}\langle\nu,f\rangle\ M_{1}(ds,d\nu) \nonumber\\
			&+\int_{0}^{t}\int_{0}^{\langle Y_{s-}, \alpha\rangle}\int_{E}\int_{0}^{p(A_{\alpha}(Y_{s-}, y), n)}\sum_{i=1}^{n}f(z_i) M(d s, d u, d w, d v).
		\end{align}
	\end{proposition}
	
	\begin{proposition}\label{prop4.3}
		For $f \in C_{0}^{1}(0, \infty)$ with the first derivative vanishing at $0$ and $F \in C^{1}[0, \infty)$, let $F_{f}(\mu)=F(\langle\mu, f\rangle)$ and let
		\begin{align*}
			\mathscr{L}_{0} F_{f}(\mu)=
			&-\langle\mu, f^{\prime}\rangle F^{\prime}(\langle\mu, f\rangle)+\int_{\mathfrak{N}(0,\infty)^{\circ}}[F(\langle\mu, f\rangle+\langle \nu, f\rangle)-F(\langle\mu, f\rangle)] L(d\nu) \\
			&-\sum_{n \in \mathbb{N}} \int_{0}^{\infty} \alpha(y) p(y, n)\int_{(0, \infty)^{\infty}}\Big[F(\langle\mu, f\rangle)-F\Big(\langle\mu,f\rangle+\sum_{i=1}^{n}f(z_i)\Big)\Big] G^{\infty}(dz)\mu(d y).
		\end{align*}
		Then we have
		\begin{align}\label{4.12}
			F_{f}(Y_{t})=F_{f}(Y_{0})+\int_{0}^{t} \mathscr{L}_{0} F_{f}(Y_{s}) d s+\text {mart}.
		\end{align}
	\end{proposition}
	
	\begin{proof}
		Let $\tilde{M}$ and $\tilde{M_{1}}$ denote the compensated measure of $M$ and $M_{1}$, respectively. By (\ref{4.3}) and It$\mathrm{\hat{o}}$'s formula we have
		\begin{small}
			\begin{align*}
				F(\langle Y_{t}, f\rangle)= F&(\langle Y_{0}, f\rangle)-\int_{0}^{t} F^{\prime}(\langle Y_{s}, f\rangle)\langle Y_{s}, f^{\prime}\rangle d s \\
				-&\int_{0}^{t}\int_{\mathfrak{N}(0,\infty)^{\circ}}[F(\langle Y_{s-}, f\rangle)-F(\langle Y_{s-}, f\rangle+\langle \nu, f\rangle)] M_{1}(ds,d\nu)\\
				-&\int_{0}^{t}\! \int_{0}^{\langle Y_{s-}, \alpha\rangle}\!\!\int_{E}\int_{0}^{p(A_{\alpha}(Y_{s-}, y), n)}\Big[F(\langle Y_{s-}, f\rangle)-F\Big(\langle Y_{s-}, f\rangle+\sum_{i=1}^{n}f(z_i)\Big)\Big] M(d s, d u, d w, d v) \\
				= F&(\langle Y_{0}, f\rangle)-\int_{0}^{t} F^{\prime}(\langle Y_{s}, f\rangle)\langle Y_{s}, f^{\prime}\rangle d s-N_{t}^{F}(f) \\
				-&\int_{0}^{t}ds\int_{\mathfrak{N}(0,\infty)^{\circ}}[F(\langle Y_{s-}, f\rangle)-F(\langle Y_{s-}, f\rangle+\langle \nu, f\rangle)] L(d\nu)\\
				-&\int_{0}^{t}\! d s\! \int_{0}^{\infty}\!\! \alpha(y) \sum_{n \in \mathbb{N}} p(y, n)\!\int_{(0, \infty)^{\infty}}\!\Big[F(\langle Y_{s-}, f\rangle)-F\Big(\langle Y_{s-}, f\rangle+\sum_{i=1}^{n}f(z_i)\Big)\Big]G^{\infty}(dz)Y_{s-}(dy)\\
				=F&(\langle Y_{0}, f\rangle)+\int_{0}^{t} \mathscr{L}_{0} F_{f}(Y_{s}) d s-N_{t}^{F}(f).
			\end{align*}
		\end{small}where
		\begin{align*}
			N_{t}^{F}(f)=&\! \int_{0}^{t}\! \int_{0}^{\langle Y_{s-}, \alpha\rangle}\!\!\! \int_{E} \int_{0}^{p(A_{\alpha}(Y_{s-}, y), n)}\!\Big[F(\langle Y_{s-}, f\rangle)-F\Big(\langle Y_{s-}, f\rangle+\sum_{i=1}^{n}f(z_i)\Big)\Big] \tilde{M}(d s, d u, d w, d v)\\
			&+\int_{0}^{t}\int_{\mathfrak{N}(0,\infty)^{\circ}}[F(\langle Y_{s-}, f\rangle)-F(\langle Y_{s-}, f\rangle+\langle \nu, f\rangle)] \tilde{M_{1}}(ds,d\nu),
		\end{align*}and $\langle X_{s}, f^{\prime} \rangle$ is well-defined since the first derivative of $f$ vanishes at $0$. A first moment estimate can check that $\{N_t^F(f): t \ge 0\}$ is a martingale. $\hfill\square$ \end{proof}

	\begin{theorem}\label{th4.1} The measure-valued process $(Y_t: t \ge 0)$ defined by $(\ref{4.2})$ is an $(\alpha,g,G,\psi)$-process. \end{theorem}

	\begin{proof} By Proposition \ref{prop4.3} we can see $(Y_{t}: t \ge 0)$ solves the martingale problem (\ref{4.12}). Let $\mathscr{F}_t=\sigma\{Y_s: 0 \le s \le t\}$. By a modification of the proof of Theorem \ref{th3.8}, we can see that
		\begin{align*}
			t \mapsto \exp \Big\{\!-\langle Y_t, u_{T-t} f\rangle-\int_0^{T-t} \psi(u_s f) d s\Big\}
		\end{align*}
		is an $(\mathscr{F}_t)$-martingale. Then $(Y_t: t \ge 0)$ is a Markov process with transition semigroup $(P_t)_{t \ge 0}$ defined by (\ref{4.1}), which completes the proof. $\hfill\square$
	\end{proof}	
	\begin{corollary}\label{cor4}
		The martingale problem (\ref{4.12}) is well-posed.
	\end{corollary}

	We now discuss the ergodicity of the $(\alpha,g,G,\psi)$-process. To this end, we need to give an estimate on the first moment of $X$. By letting $f(x)=1_{(0,\infty)}(x)$ in (\ref{2.9}) and then integrating both sides with respect to $G(dx)$, we have
	\begin{align*}
		\langle G,\pi_{t} 1 \rangle
		&=
		\int_{0}^{\infty} 1_{(0,\infty)}(x-t) G(dx) + \int_{0}^{\infty} G(dx) \int_{0}^{t} \alpha(x-s) g'(x-s,1-) \langle G,\pi_{t-s}1 \rangle ds\\
		&=
		1-G(t) + \int_{0}^{t} \langle G,\pi_{t-s}1 \rangle ds \int_{s}^{\infty} \alpha(x-s) g'(x-s,1-) G(dx).
	\end{align*}
	
	We say a nonnegative function $f$ on $[0,\infty)$ is \textit{directly Riemann integrable over $[0,\infty)$ if for any $a>0$,
		\begin{align*}
			\sum_{n=1}^{\infty} a M_n(a)<\infty\quad \text{and}\quad \lim_{a\rightarrow 0}\sum_{n=1}^{\infty} a M_n(a)=\lim_{a\rightarrow 0}\sum_{n=1}^{\infty} a m_n(a),
		\end{align*}
		where $M_n(a)=\sup\{f(x)|(n-1)a\le x\le na\}$ and $m_n(a)=\inf\{f(x)|(n-1)a\le x\le na\}$.} For the convenience of statement of the results, we formulate the following condition:

	\begin{condition}\label{cond1} (i)~ The $(\alpha,g,G,\psi)$-process is \textit{subcritical}, which means that
		\begin{align*}
			\rho_1:= \int_{0}^{\infty} ds \int_{s}^{\infty} \alpha(x-s) g'(x-s,1-) G(dx)< 1.
		\end{align*}
		
		(ii)~ There exists a constant $\alpha_{1}\in (0,\infty)$ such that the function $e^{\alpha_{1} t}[1-G(t)]$ is directly Riemann integrable over $[0,\infty)$ and
		\begin{align*}
			\int_{0}^{\infty} e^{\alpha_{1} s} ds \int_{s}^{\infty} \alpha(x-s) g'(x-s,1-) G(dx) = 1.
		\end{align*}
	\end{condition}

	\begin{example} Consider the case where $\alpha(x)\equiv \alpha>0$ and $g'(x,1-)\equiv m>0$ for $x\in (0,\infty)$. Let $G$ be the exponential distribution with parameter $\lambda>\alpha m$. Then Condition \ref{cond1} holds with $\alpha_{1}= \lambda-\alpha m$. In this case, the nonnegative function $e^{\alpha_{1} t}[1-G(t)]= e^{(\alpha_{1}-\lambda) t}$ is non-increasing and integrable over $[0,\infty)$ in the ordinary sense, then $e^{\alpha_{1} t}[1-G(t)]$ is directly Riemann integrable over $[0,\infty)$ by a general result; see, e.g., Athreya and Ney \cite[p.146]{Athreya72}.\end{example}

	The next proposition follows by the general result on defective renewal equation; see, e.g., Jagers \cite[Theorem 5.2.9]{Jagers75}.

	\begin{proposition}\label{prop4.2'} Suppose that Condition \ref{cond1} holds. Then
		\begin{align}\label{4.4'}
			\lim_{t\rightarrow \infty} e^{\alpha_1 t} \langle G,\pi_t 1 \rangle
			=
			\frac{1}{b}\int_{0}^{\infty}e^{\alpha_1 s}[1-G(s)] ds=:a_1 <\infty,
		\end{align}
		where
		\begin{align*}
			b = {\int_{0}^{\infty} s e^{\alpha_1 s} ds \int_{s}^{\infty} \alpha(x-s) g'(x-s,1-) G(dx)}.
		\end{align*}
		(Here we understand that $a_1 =0$ if $b=+\infty$.) \end{proposition}

	A sufficient condition for the ergodicity of the $(\alpha,g,G,\psi)$-process is given in the next theorem.

	\begin{theorem}\label{th4.2} Suppose that Condition \ref{cond1} holds. Then $P_t(\sigma, \cdot)$ converges as $t \rightarrow \infty$ to a probability measure $\eta$ on $\mathfrak{N}(0,\infty)$ for every $\sigma \in \mathfrak{N}(0,\infty)$ if
		\begin{align}\label{4.4}
			\int_{\mathfrak{N}(0,\infty)} L(d \nu) \int_{0}^{\infty} e^{\alpha_{1} x} \nu(dx) <\infty,
		\end{align}
		where $\alpha_{1}$ given as in Condition \ref{cond1}. In this case, the Laplace transform of $\eta$ is given by
		\begin{align}\label{4.5}
			\int_{\mathfrak{N}(0,\infty)} e^{-\langle \nu, f\rangle} \eta(d \nu)=\exp \Big\{\!-\int_0^{\infty} \psi(u_s f) d s\Big\},\quad f\in B(0,\infty)^{+}.
		\end{align}
	\end{theorem}
	
	\begin{proof} We note that under the condition \ref{cond1}, $\lim_{t\rightarrow \infty} \langle G,\pi_{t} 1\rangle = a_{1}e^{-\alpha_{1}t}$ with $a_{1}$ given as in Proposition \ref{prop4.2'}. Then there exists $T>0$ and $C_{1}\in(0,\infty)$ such that for $f\in B(0,\infty)^{+}$,
		\begin{align*}
			\langle G,\pi_{t} f\rangle \le C_{1}\|f\|e^{-\alpha_{1}t},\quad t>T.
		\end{align*}
		For any $t\in [0,T]$ and $f\in B(0,\infty)^{+}$ we have $\|\pi_{t} f\|\le e^{\beta T}\|f\|$ by Proposition \ref{prop3.11}, it follows that for any $f\in B(0,\infty)^{+}$,
		\begin{align*}
			\langle G,\pi_{t} f\rangle \le C_{2}\|f\|e^{-\alpha_{1}t},\quad t\in [0,T],
		\end{align*}
		where $C_{2}=e^{(\beta +\alpha_{1})T}\in(0,\infty)$.
		Then by (\ref{2.9}) for any $x>0$ and $f\in B(0,\infty)^{+}$ we have
		\begin{align}\label{4.6}
			\pi_{t}f(x)
			\le f(x-t)+\beta C_T \|f\| e^{-\alpha_{1} t} \int_{0}^{t\wedge x} e^{\alpha_{1} s} ds,\quad t\ge 0,
		\end{align}
		where $C_T=\max \{C_1,C_{2}\}\in (0,\infty)$.
		Suppose that (\ref{4.4}) holds. For any $f\in B(0,\infty)^{+}$ we see from Propositions \ref{prop2.2} and \ref{prop4.2'} that $\mathbf{P}_{\sigma}[\langle X_{t},f\rangle]\rightarrow 0$ as $t\rightarrow\infty$. Take any $a>\|f\|$, by (\ref{4.6}) and Proposition \ref{prop2.3} it is elementary to see that
		\begin{small}
			\begin{align*}
				\int_0^{\infty}\!\! \psi(u_s f) d s
				&=\int_0^{\infty} d s \int_{\mathfrak{N}(0,\infty)^{\circ}}(1-e^{-\langle \nu, u_s f\rangle}) L(d \nu) \\
				&\le\int_0^{\infty} d s \int_{\mathfrak{N}(0,\infty)^{\circ}}(1-e^{-\langle \nu, \pi_s f\rangle}) L(d \nu)\\
				&\le \int_{\mathfrak{N}(0,\infty)^{\circ}} L(d \nu) \int_0^{\infty} (1-e^{-\langle \nu, \pi_s f\rangle}) d s\\
				&\le \int_{\mathfrak{N}(0,\infty)^{\circ}} L(d \nu) \int_0^{\infty} \big(1-e^{-\int_{0}^{\infty} f(x-s) \nu(dx)-\beta C_{T}\|f\|\int_{0}^{\infty} \nu(dx) \int_{0}^{x}e^{-\alpha_{1}(s-r)} dr}\big) d s\\
				&= \int_{\mathfrak{N}(0,\infty)^{\circ}} L(d \nu) \int_0^{\infty} \big(1-e^{-\int_{0}^{\infty} f(x-s) \nu(dx)-\beta C_{T}\|f\|\alpha_{1}^{-1}e^{-\alpha_{1} s}\int_{0}^{\infty} (e^{\alpha_{1} x}-1) \nu(dx)}\big) d s\\
				&\le \int_{\mathfrak{N}(0,\infty)^{\circ}} L(d \nu) \int_0^{\infty} \Big(\int_{0}^{\infty} f(x-s) \nu(dx)-\beta C_{T}\|f\|\alpha_{1}^{-1}e^{-\alpha_{1} s}\int_{0}^{\infty} (e^{\alpha_{1} x}-1) \nu(dx)\Big) d s\\
				&= \|f\|\int_{\mathfrak{N}(0,\infty)^{\circ}}\! L(d \nu) \int_0^{\infty} \nu((s,\infty)) ds +
				\beta C_{T}\|f\|\alpha_{1}^{-2}\! \int_{\mathfrak{N}(0,\infty)^{\circ}} L(d \nu)\int_{0}^{\infty} (e^{\alpha_{1} x}-1) \nu(dx)\\
				&<\infty .
			\end{align*}
		\end{small}Then, as $t \rightarrow \infty$,
		\begin{align*}
			\int_0^t \psi(u_s f) d s&=\int_0^t d s \int_{\mathfrak{N}(0,\infty)^{\circ}}(1-\mathrm{e}^{-\langle \nu, u_s f\rangle}) L(d \nu) \\
			& \rightarrow \int_0^{\infty} d s \int_{\mathfrak{N}(0,\infty)^{\circ}}(1-\mathrm{e}^{-\langle \nu, u_s f\rangle}) L(d \nu)=\int_0^{\infty} \psi(u_{\mathrm{s}} f) d s<\infty .
		\end{align*}
		From (\ref{4.1}) we infer that $P_t(\sigma, \cdot)$ converges as $t \rightarrow \infty$ to a probability measure $\eta$ defined by (\ref{4.5}); see, e.g., Li \cite[Theorem 1.20]{Li22}.
		$\hfill\square$
	\end{proof}

	\section{The occupation time processes}\label{setion5}

	Let $X=(\Omega, \mathscr{F}, \mathscr{F}_{r,t}, X_{t}, \mathbf{P}_{r,\sigma})$ be a c\`{a}dl\`{a}g realization of the $(\alpha,g,G)$-process started from time $r\ge 0$. In view of (\ref{2.5}) and (\ref{2.7'}), for any $t\ge r \ge 0$ and $f\in B(0,\infty)^{+}$ we have
	\begin{align}\label{5.19}
		\mathbf{P}_{r,\sigma} \exp \{-\langle X_{t},f\rangle\}=\exp \{-\langle\sigma, v_{r}\rangle\},
	\end{align}
	where $(r,x)\mapsto v_{r}(x):=u_{t-r}f(x)$ is the unique bounded positive solution on $[0,t]\times(0,\infty)$ of
	\begin{align}\label{5.20}
		v_{r}(x)=f(x-t+r)+\int_{r}^t \alpha(x-s+r) [1-g(x-s+r,\langle G, e^{-v_s}\rangle)] d s..
	\end{align}

	\begin{proposition}\label{prop5.14} Suppose that $\{s_1<\cdots<s_n\} \subset[0, \infty)$ and $\{f_1, \ldots, f_n\} \subset B(0,\infty)^{+}$. Then we have
		\begin{align}\label{5.21}
			\mathbf{P}_{r,\sigma} \exp \Big\{\!-\sum_{j=1}^n \langle X_{s_j},f_j\rangle 1_{\{r \le s_j\}}\Big\}=\exp \{-\langle\sigma, v_{r}\rangle\},\quad 0 \le r \le s_n,
		\end{align}
		where $(r,x) \mapsto v_r(x)$ is the unique bounded positive solution on $[0, s_n] \times (0,\infty)$ of
		\begin{align}\label{5.22}
			v_{r}(x) = \sum\limits_{j=1}^n f_j(x-s_j+r) 1_{\{r \le s_j\}} +\int_{r}^{s_{n}} \alpha(x-s+r) [1-g(x-s+r,\langle G, e^{-v_s}\rangle)] ds.
		\end{align}
	\end{proposition}
	
	\begin{proof} We shall give the proof by induction in $n \ge 1$. For $n=1$ the result follows from (\ref{5.19}) and (\ref{5.20}). Now supposing (\ref{5.21}) and (\ref{5.22}) are satisfied when $n$ is replaced by $n-1$, we prove they are also true for $n$. It is clearly sufficient to consider the case with $0 \le r \le s_1<\cdots<s_n$. By the Markov property of $X$ we have
		\begin{align*}
			\mathbf{P}_{r, \sigma} \exp \Big\{\!-\sum_{j=1}^n \langle X_{s_j},f_j\rangle\Big\}
			&=\mathbf{P}_{r,\sigma}\Big[\mathbf{P}_{r,\sigma}\Big[\exp\Big\{\!-\langle X_{s_1},f_1\rangle-\sum_{j=2}^{n}\langle X_{s_j},f_j\rangle\Big\}\Big|\mathscr{F}_{r,s_1}\Big]\Big]\\
			&=\mathbf{P}_{r,\sigma}\Big[\exp\{-\langle X_{s_1},f_1\rangle\}\mathbf{P}_{r,\sigma}\Big[\exp\Big\{\!-\sum_{j=2}^{n}\langle X_{s_j},f_j\rangle\Big\}\Big|\mathscr{F}_{r,s_1}\Big]\Big]\\
			&=\mathbf{P}_{r,\sigma}\Big[\exp\{-\langle X_{s_1},f_1\rangle\}\mathbf{P}_{s_1,X_{s_1}}\Big[\exp\Big\{\!-\sum_{j=2}^{n}\langle X_{s_j},f_j\rangle\Big\}\Big]\Big]\\
			&=\mathbf{P}_{r,\sigma} \exp\{-\langle X_{s_1},f_1+\tilde{v}_{s_1}\rangle\},
		\end{align*}
		where $(r, x) \mapsto \tilde{v}_r(x)$ is a bounded positive Borel function on $[0, s_n] \times (0,\infty)$ satisfying
		\begin{align}\label{5.23}
			\tilde{v}_{r}(x) = \sum\limits_{j=2}^n f_j(x-s_j+r) 1_{\{r \le s_j\}} +\int_{s_1}^{s_{n}} \alpha(x-s+r) [1-g(x-s+r,\langle G, e^{-\tilde{v}_s}\rangle)] ds.
		\end{align}
		Then the result for $n=1$ implies that
		\begin{align*}
			\mathbf{P}_{r, \sigma} \exp \Big\{\!-\sum_{j=1}^n \langle X_{s_j},f_j\rangle\Big\}=\exp \{-\langle\sigma, v_{r}\rangle\},
		\end{align*}
		with $(r, x) \mapsto v_r(x)$ being a bounded positive Borel function on $[0, s_1] \times (0,\infty)$ satisfying
		\begin{small}
			\begin{align}\label{5.24}
				v_{r}(x)\! =\! [f_1\!(x-s_1+r)+\tilde{v}_{s_1}\!(x-s_1+r)] 1_{\{r \le s_1\}}\!+\!\int_{r}^{s_{1}} \!\!\!\alpha(x-s+r) [1-g(x-s+r,\langle G, e^{-v_s}\rangle)] ds.
			\end{align}
		\end{small}Setting $v_r=\tilde{v}_r$ for $s_1<r \le s_n$, from (\ref{5.23}) and (\ref{5.24}) one checks that $(r, x) \mapsto$ $v_r(x)$ is a bounded positive solution on $[0, s_n] \times (0,\infty)$ of (\ref{5.22}). The uniqueness of the solution of (\ref{5.22}) is a standard application of Gronwall's inequality. $\hfill\square$ \end{proof}

	\begin{proposition}\label{prop5.15} Suppose that $t\ge 0$ and $\lambda(d s)$ is a finite Borel measure on $[0, t]$. Let $(s, x) \mapsto f_s(x)$ be a bounded positive Borel function on $[0, t] \times (0,\infty)$. Then we have
		\begin{align}\label{5.25}
			\mathbf{P}_{r, \sigma} \exp \Big\{\!-\int_{[r, t]} \langle X_{s},f_s\rangle \lambda(d s)\Big\}=\exp \{-\langle\sigma, v_{r}\rangle\},\quad 0\le r\le t,
		\end{align}
		where $(r, x) \mapsto v_r(x)$ is the unique bounded positive solution on $[0, t] \times (0,\infty)$ of
		\begin{align}\label{5.26}
			v_{r}(x)= \int_{[r,t]} f_s(x-s+r) \lambda(ds)+\int_{[r,t]} \alpha(x-s+r) [1-g(x-s+r,\langle G, e^{-v_s}\rangle)] ds.
		\end{align}
	\end{proposition}
	
	\begin{proof} \textit{Step 1}. We first assume $(s, x) \mapsto f_s(x)$ is uniformly continuous on $[0, t] \times (0,\infty)$. For any integer $n \ge 1$ let $\epsilon_n(k)=k / 2^n$. From Proposition \ref{prop5.14} we see that
		\begin{small}
			\begin{align}\label{5.27}
				P_{r, \sigma} \exp \Big\{\!-\sum_{k=0}^{2^n} \langle X_{\epsilon_n(k) t},f_{\epsilon_n(k) t}\rangle \lambda\big((\epsilon_n(k) t, \epsilon_n(k+1) t] \cap [0, t]\big) 1_{\{r\le \epsilon_n(k) t\}}\Big\}=\exp\{-\langle \sigma,v_{n}(r)\rangle\},
			\end{align}
		\end{small}where $(r, x) \mapsto v_n(r, x)$ is a bounded positive solution on $[0, t] \times (0,\infty)$ to
		\begin{align}\label{5.28}
			v_{n}(r,x)=
			& \sum\limits_{k=0}^{2^n} f_{\epsilon_n(k) t}(x-\epsilon_n(k) t+r) \lambda\big((\epsilon_n(k) t, \epsilon_n(k+1) t] \cap [0, t]\big) 1_{\{r\le \epsilon_n(k) t\}}\nonumber\\
			&+\int_{r}^{\epsilon_n(k) t} \alpha(x-s+r) [1-g(x-s+r,\langle G, e^{-v_{n}(s,\cdot)}\rangle)] ds.
		\end{align}
		By letting $n\rightarrow\infty$ in (\ref{5.27}) and using the uniform continuity of $(s, x) \mapsto f_s(x)$ we have the limit $v_r(x)=\lim _{n \rightarrow \infty} v_n(r, x)$ exists and (\ref{5.25}) holds. It is not hard to show that $\{v_n\}$ is uniformly bounded on $[0, t] \times (0,\infty)$. Then we get (\ref{5.26}) by letting $n \rightarrow \infty$ in (\ref{5.28}).
		
		\textit{Step 2}. We use $B([0, t] \times (0,\infty))^{+}$ to denote the set of bounded positive Borel functions $(s,x)\mapsto f_{s}(x)$ on $[0, t] \times (0,\infty)$. Let $B_1 \subset B([0, t] \times (0,\infty))^{+}$ be the set of functions $(s, x) \mapsto f_s(x)$ for which there exist bounded positive solutions $(r, x) \mapsto v_r(x)$ of (\ref{5.26}) such that (\ref{5.25}) holds. It is easy to show that $B_1$ is closed under bounded pointwise convergence. The result of the first step shows that $B_1$ contains all uniformly continuous functions in $B([0, t] \times (0,\infty))^{+}$, so we have $B_1=B([0, t] \times (0,\infty))^{+}$ by Li \cite[Proposition 1.3]{Li22}.
		
		\textit{Step 3}. To show the uniqueness of the solution of (\ref{5.26}), suppose that $(r, x) \mapsto$ $\tilde{v}_r(x)$ is another bounded positive Borel function on $[0, t] \times (0,\infty)$ satisfying this equation. It is easy to find a constant $K_1,K_2 \ge 0$ such that
		\begin{align*}
			\|\tilde{v}_r - v_r\| \le K_1 \int_{r}^{t}\|e^{-v_s}-e^{-\tilde{v}_s}\| ds \le K_2 \int_{r}^{t}\|\tilde{v}_s - v_s\| ds.
		\end{align*}
		We may rewrite the above inequality into
		\begin{align*}
			\|\tilde{v}_{t-r}-v_{t-r}\| \le K_2 \int_{0}^{r}\|\tilde{v}_{t-s} - v_{t-s}\| ds,\quad 0 \le r \le t.
		\end{align*}
		so Gronwall's inequality implies $\|\tilde{v}_{t-r}-v_{t-r}\|=0$ for every $0 \le r \le t$. $\hfill\square$ \end{proof}
	
	\begin{proposition}\label{prop5.16} Let $t \ge 0$ be given. Let $h \in B(0,\infty)^{+}$and let $(s, x) \mapsto f_s(x)$ be a bounded positive Borel function on $[0, t] \times (0,\infty)$. Then for $0 \le r \le t$ we have
		\begin{align}\label{5.29}
			\mathbf{P}_{r, \sigma} \exp \Big\{\!-\langle X_{t},h\rangle-\int_{r}^{t} \langle X_s,f_s\rangle ds\Big\}=\exp \{-\langle\sigma, v_{r}\rangle\},
		\end{align}
		where $(r, x) \mapsto v_r(x)$ is the unique bounded positive solution on $[0, t] \times (0,\infty)$ of
		\begin{small}
			\begin{align}\label{5.30}
				v_{r}(x)= h(x-t+r)+\int_{r}^{t} f_s(x-s+r) ds+\int_{r}^{t} \alpha(x-s+r) [1-g(x-s+r,\langle G, e^{-v_s}\rangle)] ds.
			\end{align}
		\end{small}
	\end{proposition}
	
	\begin{proof} This follows by an application of Proposition \ref{prop5.15} to the measure $\lambda(ds)=ds+\delta_t(ds)$ and the function $f_s(x)=1_{\{s<t\}} f_s(x)+1_{\{s=t\}} h(x)$. $\hfill\square$ \end{proof}
	
	\begin{corollary}\label{cor5.17} Let $X=(\Omega, \mathscr{F}, \mathscr{F}_{t}, X_t, \mathbf{P}_{\sigma})$ be a c\`{a}dl\`{a}g realization of the $(\alpha,g,G)$-process started from time zero. Then for $t\ge 0$ and $f,h\in B(0,\infty)^{+}$ we have
		\begin{align}\label{5.31}
			\mathbf{P}_{\sigma} \exp \Big\{\!-\langle X_t,h\rangle-\int_0^t \langle X_s,f\rangle ds\Big\}=\exp \{-\langle \sigma,v_{t}\rangle \},
		\end{align}
		where $(t, x) \mapsto v_t(x)$ is the unique locally bounded positive solution of
		\begin{align}\label{5.32}
			v_{t}(x)= h(x-t)+\int_{0}^{t} f(x-s) ds+\int_{0}^{t} \alpha(x-s) [1-g(x-s,\langle G, e^{-v_{t-s}}\rangle)] ds.
		\end{align}
	\end{corollary}

	\begin{corollary}\label{cor5.18} Let $X=(\Omega, \mathscr{F}, \mathscr{F}_{t}, X_t, \mathbf{P}_{\sigma})$ be a c\`{a}dl\`{a}g realization of the $(\alpha,g,G)$-process started from time zero. Then for $t\ge 0$ and $f\in B(0,\infty)^{+}$ we have
		\begin{align}\label{5.31'}
			\mathbf{P}_{\sigma} \exp \Big\{\!-\int_0^t \langle X_s,f\rangle ds\Big\}=\exp \{-\langle \sigma,v_{t}f\rangle \},
		\end{align}
		where $v_t f(x)$ is the unique locally bounded positive solution of
		\begin{align}\label{5.32'}
			v_{t}f(x)= \int_{0}^{t} f(x-s) ds+\int_{0}^{t} \alpha(x-s) [1-g(x-s,\langle G, e^{-v_{t-s}f}\rangle)] ds.
		\end{align}
	\end{corollary}

	By arguments similar to those in the proofs of Propositions \ref{prop2.2} and \ref{prop2.2'}, we can obtain the following:

	\begin{proposition}\label{prop5.19} For any $t \ge 0$ and $\sigma \in \mathfrak{N}(0, \infty)$ we have
		\begin{align}\label{5.19'}
			\mathbf{P}_{\sigma}\Big[\int_{0}^{t}\langle X_s,f\rangle ds\Big]=\langle\sigma, \Pi_{t} f\rangle, \quad f\in B(0,\infty),
		\end{align}
		where $(t,x)\mapsto \Pi_{t}f(x)$ is the unique locally bounded solution of
		\begin{align}\label{5.19''}
			\Pi_{t} f(x)=\int_{0}^{t} f(x-s) ds+\int_{0}^{t} \alpha(x-s) g'(x-s,1-) \langle G,\Pi_{t-s}f\rangle ds,
		\end{align}
		which defines a semigroup of bounded kernels $(\Pi_{t})_{t \ge 0}$ on $(0,\infty)$.
	\end{proposition}
	
	\begin{proposition}\label{prop5.20} Suppose that $\|g''(\cdot,1-)\|<\infty$. Then for any $t \ge 0$ and $\sigma \in \mathfrak{N}(0,\infty)$ we have
		\begin{align}\label{5.20'}
			\mathbf{P}_{\sigma}\Big[\Big(\int_{0}^{t}\langle X_s,f\rangle ds\Big)^2\Big]=\langle\sigma, \Pi_{t} f\rangle^{2}+\langle\sigma, \Gamma_{t} f\rangle, \quad f \in B(0,\infty),
		\end{align}
		where $(\Pi_{t})_{t \ge 0}$ is defined by (\ref{5.19''}) and $(t,x)\mapsto \Gamma_{t}f(x)$ is the unique locally bounded solution of
		\begin{align}\label{5.20''}
			\Gamma_{t} f(x)
			=&\int_{0}^{t} \alpha(x-s) \big[g''(x-s,1-) \langle G,\Pi_{t-s}f\rangle^{2} + g'(x-s,1-) \langle G,(\Pi_{t-s}f)^{2}\rangle\big] ds \nonumber\\
			&+\int_{0}^{t} \alpha(x-s) g'(x-s,1-) \langle G,\Gamma_{t-s}f\rangle ds.
		\end{align}
	\end{proposition}

	In fact, as in the proofs of Propositions \ref{prop2.2} and \ref{prop2.2'}, we have, for $f\in B(0,\infty)^+$,
	\begin{align*}
		\Pi_{t} f(x)=\frac{\partial}{\partial \theta} v_{t}(\theta f)(x)\Big|_{\theta=0},
		\quad
		\Gamma_{t}f(x)=-\frac{\partial^2}{\partial \theta^2} v_{t}(\theta f)(x)\Big|_{\theta=0},
	\end{align*}
	where $\theta>0$ and $v_t f(x)$ is defined by (\ref{5.32'}).

	\begin{proposition}\label{prop5.21} For any $t\ge 0$ and $f\in B(0,\infty)$ we have
		\begin{align}\label{5.21'}
			\Pi_t f(x)=\int_{0}^{t} \pi_s f(x) ds.
		\end{align}
		where $\pi_{t} f(x)$ is the unique locally bounded solution to (\ref{2.9}).
	\end{proposition}
	
	\begin{proof} By integrating both sides of the equation (\ref{2.9}) over $(0, t]$ we have
		\begin{align*}
			\int_0^t \pi_s f(x) d s
			&=\int_0^t f(x-s) d s+\int_0^t d s \int_0^s \alpha(x-r) g'(x-r,1-)\langle G, \pi_{s-r} f\rangle d r\\
			&=\int_0^t f(x-s) d s+\int_0^t \alpha(x-r) g'(x-r,1-) d r \int_r^t \langle G, \pi_{s-r} f\rangle d s\\
			&=\int_0^t f(x-s) d s+\int_0^t \alpha(x-r) g'(x-r,1-) d r \int_{0}^{\infty} G(dx) \int_0^{t-r} \pi_{s} f(x) d s,
		\end{align*}
		where the last equality follows by Fubini's theorem. Then the uniqueness of the solution to the equation (\ref{5.19''}) we obtain (\ref{5.21'}). $\hfill\square$ \end{proof}

	Now let us consider the $(\alpha,g,G,\psi)$-process. It is a natural generalization of the case of the process without immigration. Let $Y=(\Omega, \mathscr{F}, \mathscr{F}_{r,t}, Y_{t}, \mathbf{P}_{r, \sigma})$ be a c\`{a}dl\`{a}g realization of the process started from time $r\ge 0$. In view of  (\ref{2.7'}) and (\ref{4.1}), for any $t\ge r \ge 0$ and $f\in B(0,\infty)^{+}$ we have
	\begin{align}\label{5.24'}
		\mathbf{P}_{r, \sigma} \exp \{-\langle Y_t,f\rangle\}=\exp \Big\{\!-\langle\sigma, v_{r}\rangle-\int_{r}^t \psi(v_s) d s\Big\},
	\end{align}
	where $(r,x)\mapsto v_{r}(x):=u_{t-r}f(x)$ is the unique bounded positive solution on $[0,t]\times(0,\infty)$ of (\ref{5.20}). We also omit the proofs of the following results since the arguments are similar to those for the corresponding results of the process without immigration.

	\begin{proposition}\label{prop5.14'} Suppose that $\{s_1<\cdots<s_n\} \subset[0, \infty)$ and $\{f_1, \ldots, f_n\} \subset B(0,\infty)^{+}$. Then we have
		\begin{align}\label{5.21''}
			\mathbf{P}_{r, \sigma} \exp \Big\{\!-\sum_{j=1}^n \langle Y_{s_j},f_j\rangle 1_{\{r \le s_j\}}\Big\}=\exp \Big\{\!-\langle\sigma, v_{r}\rangle-\int_{r}^{s_n} \psi(v_s) d s\Big\},\quad 0 \le r \le s_n,
		\end{align}
		where $(r, x) \mapsto v_r(x)$ is the unique bounded positive solution on $[0, s_n] \times (0,\infty)$ of (\ref{5.22}). \end{proposition}
	
	\begin{proposition}\label{prop5.15'} Suppose that $t\ge 0$ and $\lambda(d s)$ is a finite Borel measure on $[0, t]$. Let $(s, x) \mapsto f_s(x)$ be a bounded positive Borel function on $[0, t] \times (0,\infty)$. Then we have
		\begin{align}\label{5.25''}
			\mathbf{P}_{r, \sigma} \exp \Big\{\!-\int_{[r, t]} \langle Y_{s},f_s\rangle \lambda(d s)\Big\}=\exp \Big\{\!-\langle\sigma, v_{r}\rangle-\int_{r}^t \psi(v_s) d s\Big\},\quad 0\le r\le t,
		\end{align}
		where $(r, x) \mapsto v_r(x)$ is the unique bounded positive solution on $[0, t] \times (0,\infty)$ of (\ref{5.26}). \end{proposition}
	
	\begin{proposition}\label{prop5.16'} Let $t \ge 0$ be given. Let $h \in B(0,\infty)^{+}$and let $(s, x) \mapsto f_s(x)$ be a bounded positive Borel function on $[0, t] \times (0,\infty)$. Then for $0 \le r \le t$ we have
		\begin{align}\label{5.29''}
			\mathbf{P}_{r, \sigma} \exp \Big\{\!-\langle Y_{t},h\rangle-\int_{r}^{t} \langle Y_s,f_s\rangle ds\Big\}=\exp \Big\{\!-\langle\sigma, v_{r}\rangle-\int_{r}^t \psi(v_s) d s\Big\},
		\end{align}
		where $(r, x) \mapsto v_r(x)$ is the unique bounded positive solution on $[0, t] \times (0,\infty)$ of (\ref{5.30}). \end{proposition}

	\begin{corollary}\label{cor5.17'} Let $Y=(\Omega, \mathscr{F}, \mathscr{F}_{t}, Y_t, \mathbf{P}_{\sigma})$ be a c\`{a}dl\`{a}g realization of the $(\alpha,g,G,\psi)$-process started from time zero. Then for $t\ge 0$ and $f,h\in B(0,\infty)^{+}$ we have
		\begin{align}\label{5.31''}
			\mathbf{P}_{\sigma} \exp \Big\{-\langle Y_t,h\rangle- \int_0^t \langle Y_s,f\rangle ds\Big\}
			=
			\exp \Big\{\!-\langle \sigma,v_{t}\rangle -\int_{0}^{t}\psi(v_{s})ds\Big\},
		\end{align}
		where $(t, x) \mapsto v_t(x)$ is the unique locally bounded positive solution of (\ref{5.32}).
	\end{corollary}
	
	\begin{corollary}\label{cor5.18'} Let $Y=(\Omega, \mathscr{F}, \mathscr{F}_{t}, Y_t, \mathbf{P}_{\sigma})$ be a c\`{a}dl\`{a}g realization of the $(\alpha,g,G,\psi)$-process started from time zero. Then for $t\ge 0$ and $f\in B(0,\infty)^{+}$ we have
		\begin{align}\label{5.31''.}
			\mathbf{P}_{\sigma} \exp\Big\{- \int_0^t \langle Y_s,f\rangle ds\Big\}
			=
			\exp \Big\{-\langle \sigma,v_{t}f\rangle -\int_{0}^{t}\psi(v_{s}f)ds\Big\},
		\end{align}
		where $v_t f(x)$ is the unique locally bounded positive solution of (\ref{5.32'}).
	\end{corollary}

	\begin{proposition}\label{prop5.19'} For any $t \ge 0$ and $\sigma \in \mathfrak{N}(0, \infty)$ we have
		\begin{align}\label{5.19'.}
			\mathbf{P}_{\sigma}\Big[\int_0^t \langle Y_s,f\rangle ds\Big]= \langle\sigma, \Pi_{t} f\rangle+\int_{\mathfrak{N}(0,\infty)^{\circ}} L(d \nu) \int_0^t \langle \nu,\Pi_s f \rangle ds, \quad f \in B(0, \infty),
		\end{align}
		where $\Pi_{t}f(x)$ is the unique solution to (\ref{5.19''}). \end{proposition}

	\begin{proposition}\label{prop5.20'} Suppose that $\|g''(\cdot,1-)\|<\infty$. Then for any $t \ge 0$, $f \in B(0,\infty)$ and $\sigma \in \mathfrak{N}(0, \infty)$ we have
		\begin{align}\label{5.20'.}
			\mathbf{P}_{\sigma}\bigg[\Big(\int_0^t \langle Y_s,f\rangle ds\Big)^2\bigg]=
			&\Big[\langle\sigma, \Pi_{t} f\rangle+\int_{\mathfrak{N}(0,\infty)^{\circ}} L(d \nu) \int_0^t \langle \nu,\Pi_s f \rangle ds\Big]^2 + \langle\sigma, \Gamma_{t} f\rangle \nonumber\\
			&+\int_{\mathfrak{N}(0,\infty)^{\circ}} L(d \nu) \int_0^t \langle \nu,\Gamma_s f \rangle+\langle \nu,\Pi_s f \rangle^2 ds,
		\end{align}
		where $\Pi_{t}f(x)$ and $\Gamma_{t}f(x)$ are the unique solution to (\ref{5.19''}) and (\ref{5.20''}), respectively. \end{proposition}

	It is easy to see that Corollary~\ref{cor5.17'} gives a characterization of the joint distribution of random vector $(Y_t,Z_t)$, where
	\begin{align*}
		Z_t= \int_0^t Y_sds, \quad t\ge 0.
	\end{align*}
	Following Iscoe \cite{Iscoe86b}, we called $(Z_t: t \ge 0)$ the \textit{occupation time process} of the $(\alpha,g,G,\psi)$-process.

	\section{Limit Theorems}\label{setion6}

	In this section we give some limit theorems for the occupation time process $(Z_t: t \ge 0)$ under Condition \ref{cond1}. To this end, we need to give some estimates on the first and second moment of $(Y_t: t \ge 0)$ as follows.
	
	\begin{proposition}\label{prop5.22} Suppose that Condition \ref{cond1} holds. Then for any $x\in (0,\infty)$ and $f \in B(0, \infty)^{+}$, we have
		\begin{align*}
			\Pi_t f(x) \nearrow \Pi_{\infty} f(x)=\int_{0}^{\infty}\pi_{s}f(x) ds, \quad \text{as}\ t\rightarrow\infty,
		\end{align*}
		where $\pi_{t} f(x)$ is the unique solution to (\ref{2.9}) and
		\begin{align}\label{5.22'}
			\Pi_{\infty} f(x)=\int_0^x f(x-s) d s+\int_0^x \alpha(x-s) g'(x-s,1-) d s \int_0^{\infty}\langle G,\pi_r f\rangle d r.
		\end{align}
		Moreover, for any $f \in B(0, \infty)^{+}$ there exists $C_3 (f) \in (0,\infty)$ such that
		\begin{align}\label{5.22''}
			\Pi_{\infty}f(x)\le C_3 (f) e^{\alpha_{1}x},\quad x\in (0,\infty),
		\end{align}
		where $\alpha_{1}$ given as in Condition \ref{cond1}. \end{proposition}
	
	\begin{proof} Observe that (\ref{5.22'}) follows by (\ref{2.9}). Then we only need to verify the inequality (\ref{5.22''}). In fact, for any $x\in (0,\infty)$ and $f \in B(0, \infty)^{+}$, by (\ref{4.6}) we have
		\begin{align*}
			\Pi_{\infty} f(x)=\int_{0}^{\infty} \pi_t f(x) dt
			&\le \int_0^x f(x-t) dt+\beta C_T \|f\| \int_0^{\infty} e^{-\alpha_{1} t} dt \int_{0}^{x} e^{\alpha_{1} s} ds\\
			&= \int_0^x f(x-t) dt+\beta C_T \|f\| \alpha_{1}^{-2} (e^{\alpha_1 x}-1)\\
			&\le \|f\| (\alpha_{1}^{-1}+\beta C_T \alpha_{1}^{-2}) e^{\alpha_1 x}.
		\end{align*}
		Then we obtain (\ref{5.22''}) by letting $C_3 (f)=\|f\| (\alpha_{1}^{-1}+\beta C_T \alpha_{1}^{-2})$.
		$\hfill\square$
	\end{proof}

	\begin{proposition}\label{prop5.23} Suppose that Condition \ref{cond1} holds and let $(\rho_1,\alpha_1)$ be given as in the condition. In addition, assume that
		\begin{align}\label{5.23'}
			\int_{0}^{\infty} e^{2\alpha_1 x} G(dx)<\infty,
			\quad
			\|g''(\cdot,1-)\|<\infty.
		\end{align}
		Let $\Pi_{\infty} f(x)$ be given as in Proposition \ref{prop5.22} and let\begin{small}
			\begin{align}\label{5.23''.}
				a_2(f)= \frac{1}{1-\rho_1}\int_0^{\infty} G(dx)\!\!\int_0^x \alpha(x-s)[g''(x-s,1-)\langle G,\Pi_{\infty} f\rangle ^2 \!+\! g'(x-s,1-)\langle G, (\Pi_{\infty} f)^2\rangle ] d s.
			\end{align}
		\end{small}Then for any $x\in (0,\infty)$ and $f \in B(0, \infty)^{+}$ we have $\Gamma_t f(x) \rightarrow \Gamma_{\infty} f(x)$ as $t\rightarrow\infty$, where\begin{small}
			\begin{align}\label{5.23''}
				\Gamma_{\infty} f(x)=\int_0^x \alpha(x-s)\big\{g''(x-s,1-)\langle G,\Pi_{\infty} f\rangle ^2+g'(x-s,1-)\big[\langle G, (\Pi_{\infty} f)^2\rangle+a_2 (f)\big] \big\} ds.
			\end{align}
		\end{small}Moreover, for any $f \in B(0, \infty)^{+}$ there exists $C_4 (f) \in (0,\infty)$ such that
		\begin{align}\label{5.23'''}
			\Gamma_{\infty}f(x)\le C_4 (f) x,\quad x\in (0,\infty).
		\end{align}
	\end{proposition}
	
	\begin{proof}
		Integrating both sides of (\ref{5.20''}) with respect to $G(dx)$ yields
		\begin{align}\label{5.23'''.}
			\langle G,\Gamma_{t} f\rangle
			=&\int_0^{\infty}G(dx)\int_{0}^{t} \alpha(x-s) \big[g''(x-s,1-) \langle G,\Pi_{t-s}f\rangle^{2} + g'(x-s,1-) \langle G,(\Pi_{t-s}f)^{2}\rangle\big] ds \nonumber\\
			&+\int_0^{\infty}G(dx)\int_{0}^{t} \alpha(x-s) g'(x-s,1-) \langle G,\Gamma_{t-s}f\rangle ds\nonumber\\
			=&\int_0^{\infty}G(dx)\int_{0}^{t} \alpha(x-s) \big[g''(x-s,1-) \langle G,\Pi_{t-s}f\rangle^{2} + g'(x-s,1-) \langle G,(\Pi_{t-s}f)^{2}\rangle\big] ds \nonumber\\
			&+\int_{0}^{t} \langle G,\Gamma_{t-s}f\rangle ds \int_s^{\infty}\alpha(x-s) g'(x-s,1-) G(dx).
		\end{align}
		By (\ref{5.22''}) we note that
		\begin{align*}
			\langle G,(\Pi_{\infty}f)^2 \rangle = \int_0^{\infty}(\Pi_{\infty} f(x))^2 G(d x) \le C_3 (f) \int_0^{\infty} e^{2 \alpha_1 x} G(d x).
		\end{align*}
		Then by (\ref{5.23'}) we obtain
		\begin{align*}
			0
			&\le \int_0^{\infty}G(dx)\int_{0}^{t} \alpha(x-s) \big[g''(x-s,1-) \langle G,\Pi_{t-s}f\rangle^{2} + g'(x-s,1-) \langle G,(\Pi_{t-s}f)^{2}\rangle\big] ds\\
			&\le \|\alpha\|\int_0^{\infty}\! x G(dx) \big[\|g''(\cdot,1-)\|\langle G,\Pi_{\infty}f\rangle^{2} + \|g'(\cdot,1-)\| \langle G,(\Pi_{\infty}f)^{2}\rangle\big]\\
			&=:a_3 (f)<\infty.
		\end{align*}
		On the other hand, by monotone convergence we have
		\begin{align*}
			\lim_{t \rightarrow \infty} &\int_0^{\infty}G(dx)\int_{0}^{t} \alpha(x-s) \big[g''(x-s,1-) \langle G,\Pi_{t-s}f\rangle^{2} + g'(x-s,1-) \langle G,(\Pi_{t-s}f)^{2}\rangle\big] ds\\
			=& \int_0^{\infty}G(dx)\int_{0}^{x} \alpha(x-s) \big[g''(x-s,1-) \langle G,\Pi_{\infty}f\rangle^{2} + g'(x-s,1-) \langle G,(\Pi_{\infty}f)^{2}\rangle\big] ds\\
			\le&\ a_3 (f)<\infty.
		\end{align*}
		Then we have
		\begin{align*}
			\lim_{t \rightarrow \infty}\langle G,\Gamma_{t}f \rangle
			= a_2(f) \in (0,\infty),\quad f\in B(0,\infty)^+,
		\end{align*}
		where $a_2(f)$ is defined by (\ref{5.23''.}); see, e.g., Jagers \cite[Theorem 5.2.9]{Jagers75}. Then there exists $T'>0$ and $C_{5}\in(0,\infty)$ such that for $f\in B(0,\infty)^{+}$,
		\begin{align*}
			\langle G,\Gamma_{t} f\rangle \le a_2 (f) C_{5},\quad t>T'.
		\end{align*}
		On the other hand, by (\ref{5.23'''.}) it is easy to see
		\begin{align*}
			\langle G,\Gamma_{t} f\rangle \le a_3 (f) + \beta \int_0^t \langle G,\Gamma_{s} f\rangle ds.
		\end{align*}
		where $\beta=\|\alpha g'(\cdot,1-)\|$. Then Gronwall's inequality implies
		\begin{align*}
			\langle G,\Gamma_{t} f\rangle
			&\le a_3 (f) + \beta\ a_3 (f) \int_0^t e^{\beta(t-s)} ds\\
			&=a_3 (f) e^{\beta t}\\
			&\le a_3 (f) e^{\beta T'},\quad t\in [0,T'].
		\end{align*}
		Then for any $f\in B(0,\infty)^+$ we have
		\begin{align*}
			\langle G,\Gamma_{t} f\rangle\le max\big\{a_2 (f) C_{5}, a_3 (f) e^{\beta T'}\big\}=:C_{T'}(f)\in (0,\infty),\quad t\ge 0.
		\end{align*}
		Then letting $t\rightarrow\infty$ in (\ref{5.20''}) we obtain (\ref{5.23''}) by monotone convergence and dominated convergence. Moreover, for any $x\in (0,\infty)$ and $f \in B(0, \infty)^{+}$, by (\ref{5.23''}) we have
		\begin{align*}
			\Gamma_{\infty} f(x)\le x\cdot \|\alpha\| \big[\|g''(\cdot,1-)\|\langle G,\Pi_{\infty}f\rangle^{2} + \|g'(\cdot,1-)\| \big(\langle G,(\Pi_{\infty}f)^{2}\rangle + C_{T'}(f)\big)\big]
		\end{align*}
		Then by letting $C_4 (f)=\|\alpha\| \big[\|g''(\cdot,1-)\|\langle G,\Pi_{\infty}f\rangle^{2} + \|g'(\cdot,1-)\| \big(\langle G,(\Pi_{\infty}f)^{2}\rangle + C_{T'}(f)\big)\big]$ we naturally obtain (\ref{5.23'''}).
		$\hfill\square$
	\end{proof}
	
	Now we discuss the longtime behavior of the occupation time process $(Z_t: t \ge 0)$. The proofs are based on the moment estimates given in Proposition \ref{prop5.22} and \ref{prop5.23}.

	\begin{theorem}\label{th5.1} Supposed that Condition \ref{cond1} and the properties (\ref{4.4}) and (\ref{5.23'}) hold. Then for any $\sigma \in \mathfrak{N}(0, \infty)$ and $f \in B(0, \infty)$ we have, as $t\rightarrow\infty$,
		\begin{align*}
			\frac{1}{t}\langle Z_{t},f\rangle\rightarrow \int_{\mathfrak{N}(0,\infty)^{\circ}}\langle \nu, \Pi_{\infty}f\rangle L(d \nu)
		\end{align*}
		by convergence in probability with respect to $\mathbf{P}_{\sigma}$, where $\Pi_{\infty}f(x)$ is given as in Proposition \ref{prop5.22}. \end{theorem}

	\begin{proof} It suffices to consider $f\in B(0, \infty)^{+}$. For the convenience of statement of the proof, we write $f_t = f/t$ for $t\ge 0$. By (\ref{5.31''.}) we have
		\begin{align}\label{5.33}
			\mathbf{P}_{\sigma} \exp \Big\{\!- \frac{1}{t} \langle Z_t,f\rangle \Big\}=\exp \Big\{\!-\langle \sigma,v_{t}f_t\rangle -\int_{0}^{t}\psi(v_{s}f_t)ds\Big\}.
		\end{align}
		By Jensen's inequality and (\ref{5.22''}) we see, as $t\rightarrow\infty$,
		\begin{align*}
			v_t f_t(x)\le \Pi_t f_t(x)=\frac{1}{t}\Pi_t f(x)\le \frac{1}{t}\Pi_{\infty} f(x)\le \frac{1}{t}C_3 (f) e^{\alpha_1 x}\rightarrow 0.
		\end{align*}
		Then $\langle \sigma,v_{t}f_t\rangle$ converges as $t\rightarrow\infty$ to $0$ for any $\sigma \in \mathfrak{N}(0, \infty)$ and $f \in B(0, \infty)^{+}$ by dominated convergence. On the other hand, for any $0< s\le 1$ and $x\in(0,\infty)$, by Proposition \ref{prop5.22}, Proposition \ref{prop5.23} and Taylor's expansion we have, as $t\rightarrow\infty$,
		\begin{align*}
			tv_{st}f_t (x)
			&=tv_{st}(t^{-1}f) (x)\\
			&=\Pi_{st}f(x) -\frac{1}{2t}\Gamma_{st}f(x) + o\Big(\frac{1}{t^2}\Big)\\
			&\nearrow \Pi_{\infty} f(x).
		\end{align*}
		Then by (\ref{4.4}) we obtain, as $t\rightarrow\infty$,
		\begin{align*}
			\int_{0}^{t}\psi(v_{s}f_t)ds
			&=\int_{\mathfrak{N}(0,\infty)^{\circ}} L(d\nu) \int_0^t\big[1-e^{-\langle \nu, v_s f_t\rangle}\big] d s\\
			&=\int_0^1 ds \int_{\mathfrak{N}(0,\infty)^{\circ}} t\big[1-e^{-t^{-1}\langle \nu, t v_{st} f_t\rangle}\big] L(d\nu)\\
			&\nearrow \int_0^1 ds \int_{\mathfrak{N}(0,\infty)^{\circ}} \langle \nu, \Pi_{\infty}f \rangle L(d\nu)\\
			&=\int_{\mathfrak{N}(0,\infty)^{\circ}} \langle \nu, \Pi_{\infty}f\rangle L(d\nu)<\infty.
		\end{align*}
		By letting $t\rightarrow\infty$ in (\ref{5.33}) we obtain the desired result. $\hfill\square$
	\end{proof}

	\begin{theorem}\label{th5.2} Supposed that Condition \ref{cond1} and the property (\ref{5.23'}) hold. Let $\alpha_{1}$ be given as in Condition \ref{cond1} and assume further that
		\begin{align}\label{4.4''}
			\int_{\mathfrak{N}(0,\infty)} L(d \nu) \Big(\int_{0}^{\infty} e^{\alpha_{1} x} \nu(dx)\Big)^2 <\infty.
		\end{align}
		Then for any $\sigma \in \mathfrak{N}(0, \infty)$ and $f \in B(0, \infty)$ we have, almost surely with respect to $\mathbf{P}_{\sigma}$,
		\begin{align}\label{4.25}
			\frac{1}{t}\langle Z_{t},f\rangle\rightarrow \int_{\mathfrak{N}(0,\infty)^{\circ}}\langle \nu, \Pi_{\infty}f\rangle L(d \nu),\quad \text{as}\ t\rightarrow\infty,
		\end{align}
		where $\Pi_{\infty}f(x)$ is given as in Proposition \ref{prop5.22}. \end{theorem}

	\begin{proof}  It suffices to consider $f\in B(0, \infty)^{+}$. By Propositions \ref{prop5.19'} and \ref{prop5.20'} we have \begin{small}
			\begin{align*}
				\text{Var}[\langle Z_t,f\rangle]
				&:=
				\mathbf{P}_{\sigma}[\langle Z_t,f\rangle^2 ]-\mathbf{P}_{\sigma}[\langle Z_t,f\rangle ]^2\\
				&=
				\langle\sigma, \Gamma_{t} f\rangle+\int_{\mathfrak{N}(0,\infty)^{\circ}} L(d \nu) \int_0^t \langle \nu,\Gamma_s f \rangle+\langle \nu,\Pi_s f \rangle^2 ds\\
				&\le
				C_4 (f) \int_{0}^{\infty} x \sigma(dx)+t \int_{\mathfrak{N}(0,\infty)^{\circ}} L(d \nu) \Big[C_4 (f) \int_0^{\infty}x \nu(dx) +C_3 (f)^2\Big(\int_0^{\infty}e^{\alpha_1 x} \nu(dx)\Big)^2\Big]\\
				&=
				C_6 (f) +C_7 (f) t,
			\end{align*}
		\end{small}where the first inequality follows by Propositions \ref{prop5.22} and \ref{prop5.23},
		\begin{align*}
			C_6 (f) :=C_4 (f) \int_{0}^{\infty} x \sigma(dx)<\infty
		\end{align*}
		and
		\begin{align*}
			C_7 (f):=\int_{\mathfrak{N}(0,\infty)^{\circ}} L(d \nu) \Big[C_4 (f) \int_0^{\infty}x \nu(dx) +C_3 (f)^2 \Big(\int_{0}^{\infty} e^{\alpha_{1} x} \nu(dx)\Big)^2\Big]<\infty
		\end{align*}
		by (\ref{4.4''}). Let $a\in (1,\infty)$ and $k_n=a^{n}, n\in \mathbb{N}\backslash 0$. For any $\sigma \in \mathfrak{N}(0, \infty)$, $f \in B(0, \infty)^{+}$ and $\epsilon>0$ by Chebyshev's inequality we have
		\begin{align*}
			\sum_{n=1}^{\infty} \mathbf{P}_{\sigma}\Big(\frac{|\langle Z_{k_n},f \rangle-\mathbf{P}_{\sigma}[\langle Z_{k_n},f \rangle]|}{k_n}>\epsilon\Big)
			&\le \frac{1}{\epsilon^2}\sum_{n=1}^{\infty} \frac{1}{k_n^2} \text{Var}[\langle Z_{k_n},f\rangle]\\
			&\le \frac{1}{\epsilon^2}\sum_{n=1}^{\infty} \Big(\frac{1}{k_n^2}C_6 (f) +\frac{1}{k_n}C_7 (f) \Big) \\
			&<\infty.
		\end{align*}
		Notice that $\int_{\mathfrak{N}(0,\infty)^{\circ}}\langle \nu, \Pi_{\infty}f\rangle L(d \nu)<\infty$ for $f \in B(0, \infty)^{+}$ by (\ref{4.4}). By Theorem \ref{th4.2} it is easy to see that
		\begin{small}
			\begin{align*}
				\lim_{n\rightarrow\infty}k_n^{-1} \mathbf{P}_{\sigma}[\langle Z_{k_n},f \rangle]-\!\int_{\mathfrak{N}(0,\infty)^{\circ}}\!\langle \nu, \Pi_{\infty}f\rangle L(d \nu)
				&=\lim_{n\rightarrow\infty}k_n^{-1} \int_{0}^{k_n}\mathbf{P}_{\sigma}[\langle Y_{s},f \rangle] ds-\!\int_{\mathfrak{N}(0,\infty)^{\circ}}\!\langle \nu, \Pi_{\infty}f\rangle L(d \nu)\\
				&=\lim_{n\rightarrow\infty}\mathbf{P}_{\sigma}[\langle Y_{k_n},f \rangle]-\!\int_{\mathfrak{N}(0,\infty)^{\circ}}\!\langle \nu, \Pi_{\infty}f\rangle L(d \nu)\\
				&=0.
			\end{align*}
		\end{small}Therefore by Borel-Cantelli Lemma,
		\begin{align*}
			\frac{\langle Z_{k_n},f\rangle}{k_n}\rightarrow\int_{\mathfrak{N}(0,\infty)^{\circ}}\langle \nu, \Pi_{\infty}f\rangle L(d \nu),\quad \text{as}\ n\rightarrow\infty,
		\end{align*}
		almost surely with respect to $\mathbf{P}_{\sigma}$. For fixed $t \ge 1$, there exists $n$ such that $a^n \le t<a^{n+1}$ and
		\begin{align*}
			\frac{\langle Z_{a^n}, f\rangle}{a^{n+1}}<\frac{\langle Z_t, f\rangle}{t}<\frac{\langle Z_{a^{n+1}}, f\rangle}{a^n}.
		\end{align*}
		Then for any $a>1$,
		\begin{align*}
			\frac{1}{a} \int_{\mathfrak{N}(0,\infty)^{\circ}}\langle \nu, \Pi_{\infty}f\rangle L(d \nu) \le \liminf_{t \rightarrow \infty} \frac{\langle Z_t, f\rangle}{t} \le \limsup_{t \rightarrow \infty} \frac{\langle Z_t, f\rangle}{t} \le a \int_{\mathfrak{N}(0,\infty)^{\circ}}\langle \nu, \Pi_{\infty}f\rangle L(d \nu).
		\end{align*}
		By letting $a\rightarrow 1$ we obtain the desired result. $\hfill\square$ \end{proof}

	\begin{theorem}\label{th5.3} Supposed that Condition \ref{cond1} and the properties (\ref{5.23'}) and (\ref{4.4''}) hold. Then for any $\sigma \in \mathfrak{N}(0, \infty)$ and $f \in B(0, \infty)$, the distribution of
		\begin{align*}
			\frac{1}{\sqrt{t}}\big(\langle Z_{t},f\rangle-\mathbf{P}_{\sigma}\langle Z_{t},f\rangle\big)
		\end{align*}
		under $\mathbf{P}_{\sigma}$ converges as $t\rightarrow\infty$ to the Gaussian distribution with mean zero and variance
		\begin{align*}
			\int_{\mathfrak{N}(0,\infty)^{\circ}} \big(\langle \nu,\Gamma_{\infty} f \rangle+\langle \nu ,\Pi_{\infty} f \rangle^2\big) L(d \nu),
		\end{align*}
		where $\Pi_{\infty}f(x)$ and $\Gamma_{\infty} f(x)$ are given as in Proposition \ref{prop5.22} and \ref{prop5.23}, respectively. \end{theorem}

	\begin{proof} It suffices to consider $f\in B(0, \infty)^{+}$. For the convenience of statement of the proof, we write $f_{\sqrt{t}} = f/\sqrt{t}$ for $t\ge 0$. By (\ref{5.31''.}) and (\ref{5.19'.}) we have
		\begin{align}\label{5.34}
			\mathbf{P}_{\sigma} \exp \Big\{\!- \frac{1}{\sqrt{t}}\big[\langle Z_t,f\rangle-\mathbf{P}_{\sigma}\langle Z_t ,f\rangle \big]\Big\}=\exp \Big\{\!\langle \sigma, \Pi_{t} f_{\sqrt{t}}-v_{t}f_{\sqrt{t}}\rangle+\int_{0}^{t} H_{s}(f_{\sqrt{t}})
			ds\Big\},
		\end{align}
		where
		\begin{align*}
			H_{s}(f_{\sqrt{t}})=\int_{\mathfrak{N}(0,\infty)^{\circ}} \big[e^{-\langle \nu ,v_s f_{\sqrt{t}} \rangle}-1+\langle \nu ,\Pi_s f_{\sqrt{t}} \rangle\big] L(d \nu).
		\end{align*}
		By Jensen's inequality and (\ref{5.22''}) it is easy to see that, as $t\rightarrow \infty$,
		\begin{align*}
			\langle \sigma , v_t f_{\sqrt{t}}\rangle\le\langle \sigma , \Pi_t f_{\sqrt{t}}\rangle=\frac{1}{\sqrt{t}}\langle \sigma , \Pi_t f\rangle \le \frac{1}{\sqrt{t}}\langle\sigma , \Pi_{\infty} f\rangle \rightarrow 0.
		\end{align*}
		On the other hand, for any $0< s\le t$ and $x\in(0,\infty)$, by Taylor's expansion we have, as $t\rightarrow\infty$,
		\begin{align*}
			v_{s}f_{\sqrt{t}} (x)
			&=
			v_{s}(t^{-1/2}f) (x)\\
			&=
			\frac{1}{\sqrt{t}}\Pi_{s}f(x) -\frac{1}{2t}\Gamma_{s}f(x)+ o\Big(\frac{1}{t}\Big).
		\end{align*}
		Then by Proposition \ref{prop5.22}, Proposition \ref{prop5.23} and (\ref{4.4''}) we obtain, as $t\rightarrow\infty$,
		\begin{small}
			\begin{align*}
				&\int_0^t H_{s}(f_{\sqrt{t}}) ds\\
				&=
				\int_0^t ds \int_{\mathfrak{N}(0,\infty)^{\circ}} \big[e^{-\langle \nu ,v_s f_{\sqrt{t}} \rangle}-1+\langle \nu ,\Pi_s f_{\sqrt{t}} \rangle\big] L(d \nu)\\
				&=
				\int_0^t ds \int_{\mathfrak{N}(0,\infty)^{\circ}} \big[\langle \nu ,\Pi_s f_{\sqrt{t}}-v_s f_{\sqrt{t}} \rangle+\frac{1}{2}\langle \nu ,v_s f_{\sqrt{t}} \rangle^2\big] L(d \nu)+\int_0^t ds\int_{\mathfrak{N}(0,\infty)^{\circ}} o\big(\langle \nu ,v_s f_{\sqrt{t}} \rangle^2\big) L(d \nu)\\
				&=
				\frac{1}{2t}\int_0^t ds \int_{\mathfrak{N}(0,\infty)^{\circ}} \big[\langle \nu ,\Gamma_s f\rangle +\langle \nu ,\Pi_s f \rangle^2\big] L(d \nu)+o\big(\frac{1}{t}\big)\int_0^t ds\int_{\mathfrak{N}(0,\infty)^{\circ}} \langle \nu ,\Pi_s f \rangle^2 L(d \nu)\\
				&\rightarrow
				\frac{1}{2} \int_{\mathfrak{N}(0,\infty)^{\circ}} \big[\langle \nu ,\Gamma_{\infty} f\rangle +\langle \nu ,\Pi_{\infty} f \rangle^2\big] L(d \nu)<\infty.
			\end{align*}
		\end{small}Combining the above gives
		\begin{align*}
			\lim_{t \rightarrow \infty} \mathbf{P}_{\sigma} \exp \Big\{\!- \frac{1}{\sqrt{t}}[\langle Z_t,f\rangle-\mathbf{P}_{\sigma}\langle Z_t ,f\rangle ]\Big\}=\exp \Big\{\frac{1}{2} \int_{\mathfrak{N}(0,\infty)^{\circ}} \big[\langle \nu ,\Gamma_{\infty} f\rangle +\langle \nu ,\Pi_{\infty} f \rangle^2\big] L(d \nu)\Big\}.
		\end{align*}
		Then the theorem follows. $\hfill\square$ \end{proof}

	\acknowledgements{\rm We would like to thank the referees for their careful reading of the paper and helpful comments. We are grateful to the Laboratory of Mathematics and Complex Systems (Ministry of Education) for providing us the research facilities. This research is supported by the National Key R$\&$D Program of China (No. 2020YFA0712901).}
	
	\bf{Conflict of interest statement}\quad \rm The authors declared that they have no conflict of interest.

\end{document}